\documentclass[11pt,xcolor=pdftex,a4paper,oneside,centertags,reqno,endnotes]{amsart}

\usepackage{apptools}
\usepackage{chngcntr}

\usepackage{endnotes}
\usepackage{latexsym}
\usepackage{amsmath}
\usepackage{amssymb}   
\usepackage{epsfig}
\usepackage{amsthm}
\usepackage{enumerate}
\usepackage[utf8x]{inputenc}
\usepackage{graphics}

\usepackage{hyperref}
\usepackage[non-compressed-cites,initials,nobysame]{amsrefs}

\usepackage[normalem]{ulem}

\newcommand\rmi{\hbox{\rm (i)}}
\newcommand\rmii{\hbox{\rm (ii)}}

\usepackage{pdfsync}
\usepackage{color}
\usepackage{mathrsfs}

\usepackage[margin=1.3in]{geometry}

\usepackage[T1]{fontenc}
\usepackage{mathrsfs}
\usepackage{epsfig}
\usepackage{xypic}
\usepackage{verbatim}
\usepackage{multicol}
\usepackage{wrapfig}
\usepackage{pgf,tikz}
\usepackage{mathrsfs}
\usetikzlibrary{arrows}

\numberwithin{equation}{section}

\newtheorem{theorem}{Theorem}[section]
\newtheorem{proposition}[theorem]{Proposition}
\newtheorem{lemma}[theorem]{Lemma}
\newtheorem{corollary}[theorem]{Corollary}
\theoremstyle{definition}
\newtheorem{definition}[theorem]{Definition}
\newtheorem{remark}[theorem]{Remark}%{\emph{\bf{Remark}}}

\newtheorem*{notation*}{Notation}

\renewcommand{\Re}{\mathrm{Re}\,}

\newcommand{\Dom}{\mathrm{Dom}}

\newcommand{\N}{\mathbb{N}}

\newcommand{\R}{\mathbb{R}}
\newcommand{\C}{\mathbb{C}}

\usepackage{textcomp}
\usepackage{dsfont}
\usepackage{latexsym}
\usepackage{amssymb}
\usepackage{amsthm}
\usepackage{amsmath}
\DeclareMathAlphabet{\mathpzc}{OT1}{pzc}{m}{en}
\usepackage{yfonts}
\usepackage{xfrac}
\usepackage{fge}

\usepackage{fancyhdr}

\usepackage{newlfont}
\usepackage{graphicx}

\usepackage{mathtools}
\usepackage{comment}
\usepackage{indentfirst}
\usepackage{braket}
\usepackage{scalerel}

\DeclareMathOperator{\supp}{supp}

\DeclarePairedDelimiter{\abs}{\lvert}{\rvert}

{\left\lbrace\begin{array}{@{}l@{}}}%
{\end{array}\right.}

\makeatletter
\newcommand*{\mint}[1]{%
	\mint@l{#1}{}%
}
\newcommand*{\mint@l}[2]{%
	\@ifnextchar\limits{%
		\mint@l{#1}%
	}{%
	\@ifnextchar\nolimits{%
		\mint@l{#1}%
	}{%
	\@ifnextchar\displaylimits{%
		\mint@l{#1}%
	}{%
	\mint@s{#2}{#1}%
}%
}%
}%
}
\newcommand*{\mint@s}[2]{%
	\@ifnextchar_{%
		\mint@sub{#1}{#2}%
	}{%
	\@ifnextchar^{%
		\mint@sup{#1}{#2}%
	}{%
	\mint@{#1}{#2}{}{}%
}%
}%
}
\def\mint@sub#1#2_#3{%
	\@ifnextchar^{%
		\mint@sub@sup{#1}{#2}{#3}%
	}{%
	\mint@{#1}{#2}{#3}{}%
}%
}
\def\mint@sup#1#2^#3{%
	\@ifnextchar_{%
		\mint@sub@sup{#1}{#2}{#3}%
	}{%
	\mint@{#1}{#2}{}{#3}%
}%
}
\def\mint@sub@sup#1#2#3^#4{%
	\mint@{#1}{#2}{#3}{#4}%
}
\def\mint@sup@sub#1#2#3_#4{%
	\mint@{#1}{#2}{#4}{#3}%
}
\newcommand*{\mint@}[4]{%
	\mathop{}%
	\mkern-\thinmuskip
	\mathchoice{%
		\mint@@{#1}{#2}{#3}{#4}%
		\displaystyle\textstyle\scriptstyle
	}{%
	\mint@@{#1}{#2}{#3}{#4}%
	\textstyle\scriptstyle\scriptstyle
}{%
\mint@@{#1}{#2}{#3}{#4}%
\scriptstyle\scriptscriptstyle\scriptscriptstyle
}{%
\mint@@{#1}{#2}{#3}{#4}%
\scriptscriptstyle\scriptscriptstyle\scriptscriptstyle
}%
\mkern-\thinmuskip
\int#1%
\ifx\\#3\\\else_{#3}\fi
\ifx\\#4\\\else^{#4}\fi  
}
\newcommand*{\mint@@}[7]{%
	\begingroup
	\sbox0{$#5\int\m@th$}%
	\sbox2{$#5\int_{}\m@th$}%
	\dimen2=\wd0 %
	\let\mint@limits=#1\relax
	\ifx\mint@limits\relax
	\sbox4{$#5\int_{\kern1sp}^{\kern1sp}\m@th$}%
	\ifdim\wd4>\wd2 %
	\let\mint@limits=\nolimits
	\else
	\let\mint@limits=\limits
	\fi
	\fi
	\ifx\mint@limits\displaylimits
	\ifx#5\displaystyle
	\let\mint@limits=\limits
	\fi
	\fi
	\ifx\mint@limits\limits
	\sbox0{$#7#3\m@th$}%
	\sbox2{$#7#4\m@th$}%
	\ifdim\wd0>\dimen2 %
	\dimen2=\wd0 %
	\fi
	\ifdim\wd2>\dimen2 %
	\dimen2=\wd2 %
	\fi
	\fi
	\rlap{%
		$#5%
		\vcenter{%
			\hbox to\dimen2{%
				\hss
				$#6{#2}\m@th$%
				\hss
			}%
		}%
		$%
	}%
	\endgroup
}

\makeatletter
\def\overbracket#1{\mathop{\vbox{\ialign{##\crcr\noalign{\kern3\p@}
				\downbracketfill\crcr\noalign{\kern3\p@\nointerlineskip}
				$\hfil\displaystyle{#1}\hfil$\crcr}}}\limits}
\def\underbracket#1{\mathop{\vtop{\ialign{##\crcr
				$\hfil\displaystyle{#1}\hfil$\crcr\noalign{\kern3\p@\nointerlineskip}
				\upbracketfill\crcr\noalign{\kern3\p@}}}}\limits}
\def\overparenthesis#1{\mathop{\vbox{\ialign{##\crcr\noalign{\kern3\p@}
				\downparenthfill\crcr\noalign{\kern3\p@\nointerlineskip}
				$\hfil\displaystyle{#1}\hfil$\crcr}}}\limits}
\def\underparenthesis#1{\mathop{\vtop{\ialign{##\crcr
				$\hfil\displaystyle{#1}\hfil$\crcr\noalign{\kern3\p@\nointerlineskip}
				\upparenthfill\crcr\noalign{\kern3\p@}}}}\limits}
\def\downparenthfill{$\m@th\braceld\leaders\vrule\hfill\bracerd$}
\def\upparenthfill{$\m@th\bracelu\leaders\vrule\hfill\braceru$}
\def\upbracketfill{$\m@th\makesm@sh{\llap{\vrule\@height3\p@\@width.7\p@}}%
	\leaders\vrule\@height.7\p@\hfill
	\makesm@sh{\rlap{\vrule\@height3\p@\@width.7\p@}}$}
\def\downbracketfill{$\m@th
	\makesm@sh{\llap{\vrule\@height.7\p@\@depth2.3\p@\@width.7\p@}}%
	\leaders\vrule\@height.7\p@\hfill
	\makesm@sh{\rlap{\vrule\@height.7\p@\@depth2.3\p@\@width.7\p@}}$}
\makeatother

\newcommand{\loc}{\mathrm{loc}}
\newcommand{\glob}{\mathrm{glob}}

\newcommand{\dd}{\mathrm{d}}

\newcommand{\As}{\mathcal{A}}
\newcommand{\Ls}{\mathcal{L}}
\newcommand{\Rs}{\mathcal{R}}

\newcommand{\Pc}{\mathcal{P}}

\newcommand{\fin}{\mathrm{fin}}

\usepackage{multirow}

\usepackage[margin=1.3in]{geometry}
\usepackage{mathptmx}
\DeclareMathAlphabet{\mathcal}{OMS}{cmsy}{m}{n}
\usepackage{caption}
\begin{document}
\title[Singular integrals and Hardy type spaces for the inverse Gauss measure]{Singular integrals and Hardy type spaces for\\ the inverse Gauss measure}
\author{Tommaso Bruno}
\address{Dipartimento di Scienze Matematiche ``Giuseppe Luigi Lagrange'', Politecnico di Torino, Corso Duca degli Abruzzi 24, 10129 Torino, Italy}
\email{tommaso.bruno@polito.it}


\thanks{
%Both authors are  partially supported by the grant PRIN 2010-11 {\em Real and Complex Manifolds: Geometry, Topology and Harmonic Analysis} and the grant PRIN 2015 {\em Real and Complex Manifolds: Geometry, Topology and Harmonic Analysis}. 
    The author is member of the Gruppo Nazionale per l'Analisi Matematica, la Probabilit\`a e le loro Applicazioni (GNAMPA) of the Istituto Nazionale di Alta Matematica (INdAM)
    }

\maketitle

\begin{abstract}
Let $\gamma_{-1}$ be the absolutely continuous measure on $\R^n$ whose density is the reciprocal of a Gaussian and consider the natural weighted Laplacian $\As$ on $L^2(\gamma_{-1})$. In this paper, we prove boundedness and unboundedness results for the purely imaginary powers and the first order Riesz transforms associated with the translated operators $\As+\lambda I$, $\lambda\geq0$, from certain new Hardy-type spaces adapted to $\gamma_{-1}$ to $L^1(\gamma_{-1})$. We also investigate the weak type $(1,1)$ of these operators. 
\end{abstract}
\section{Introduction}
Let $n\in \N$ and denote with $\gamma_{-1}$ the absolutely continuous measure on $\R^n$ whose density is the reciprocal of a normalised Gaussian, i.e.\
\[\dd \gamma_{-1}(x)\coloneqq \pi^{n/2} e^{|x|^2}\, \dd x \]
where $\dd x$ is the Lebesgue measure on $\R^n$. We call $\gamma_{-1}$ the ``inverse Gauss'' measure. Consider the second-order differential operator
\[\mathcal{A}_0f(x)\coloneqq -\frac{1}{2}\Delta f(x) - x\cdot \nabla f(x), \qquad f\in C_c^\infty(\R^n).\]
It is easy to see that $\As_0$ is positive and symmetric on $L^2(\gamma_{-1})$. By a classical argument (see e.g.~\cite{Strichartz}) $\As_0$ is essentially self-adjoint on $L^2(\gamma_{-1})$; we denote with $\As$ its positive self-adjoint closure. In this paper, we prove endpoint results for the imaginary powers and the first order Riesz transforms associated with $\As$ or with its translations. By these, we mean the operators
\begin{equation}\label{translatedOp}
(\As+\lambda I)^{iu}, \quad u\in \R\setminus \{0\}, \qquad  \Rs_\lambda \coloneqq \nabla (\As +\lambda I)^{-1/2}, 
\end{equation}
for any $\lambda\geq 0$, respectively. As we shall see, it is rather natural to introduce these \emph{translations} of $\As$. The operator $\As$ was first introduced in F.~Salogni's PhD thesis~\cite{Salogni}, where the Mehler semigroup $e^{-t\As}$ is studied on $L^p(\gamma_{−1})$, $p\in (1,\infty)$, and the weak type $(1,1)$ of the associated maximal operator $\mathcal{H}^∗$ is proved (see also~\cite{MS}).

\smallskip

The interest in imaginary powers and Riesz transforms of $\As$ comes from different aspects. As pointed out by Salogni~\cite{Salogni}, the operator $\As$ can be seen as a restriction of the Laplace-Beltrami operator on a warped-product manifold whose Ricci tensor is unbounded from below. In a recent series of papers by Mauceri, Meda and Vallarino~\cite{MMVHardy,MMVAtomic, MMV, MMV-Harmonic} a theory of Hardy-type spaces on certain noncompact manifolds is developed to obtain endpoint estimates for the imaginary powers and the Riesz transforms associated with the Laplace-Beltrami operator $\mathscr{L}$ of the manifold, provided the manifold has positive injectivity radius, $\mathscr{L}$ has spectral gap, and the Ricci tensor is \emph{bounded from below}. We emphasize that such assumptions force the Riemannian measure of the manifold to be non-doubling. The study of imaginary powers and Riesz transforms of $\As$ may thus be a first step in understanding a similar theory on manifolds whose Ricci tensor is not bounded from below. 

In addition to this, we may look at the operator $\As$ as the weighted Laplacian of the weighted manifold $(\R^n,\gamma_{-1})$, i.e.\ the second order differential operator which is self-adjoint on $L^2(\gamma_{-1})$. On weighted Riemannian manifolds, there is a natural notion of curvature tensor known as \emph{Bakry-Emery} curvature~\cite{Bakry}. It was shown by Bakry~\cite{Bakry} (see also~\cite{CarbDrag}) that in the general setting of weighted Riemannian manifolds with weighted Laplacian $\mathscr{L}$, the lower bound of the Bakry-Emery curvature tensor, when it exists, plays a role in the $L^p$-boundedness, $1<p<\infty$, of the Riesz transforms of $\mathscr{L}$. No similar result is known for the endpoint at $p=1$. Under this point of view, it is noteworthy that the weighted manifold $(\R^n,\gamma_{-1})$ has constant Bakry-Emery curvature tensor equal to minus twice the identity, so that the operator $\As$ is the prototype of weighted Laplacian on a weighted manifold with constant and negative Bakry-Emery curvature tensor.

Finally, as explained by the authors in~\cite{LiSjogrenWu}, very few endpoint results are known for singular integrals associated with weighted Laplacians on manifolds with exponential, or super-exponential, volume growth. This should make our results interesting in their own.

\bigskip

The endpoint estimates we consider are the weak type $(1,1)$, i.e.\ the boundedness from $L^1(\gamma_{-1})$ to $L^{1,\infty}(\gamma_{-1})$, and the boundedness from some atomic Hardy spaces to $L^1(\gamma_{-1})$. Since $\gamma_{-1}$ is non-doubling on Euclidean balls, the measure space $(\R^n, \gamma_{-1})$ endowed with the Euclidean metric $d_{Euc}$ is not a space of homogeneous type in the sense of Coifman and Weiss~\cite{CoifmanWeiss}. Thus, the notion of Hardy space has to be suitably interpreted.

On the one hand, if we consider the metric $\rho$ defined by the length element
\[
\dd s^2=(1+|x|)^2(\dd x_1^2+\dots+\dd x_n^2)
\]
on $\R^n$, the metric measure space $(\R^n, \rho, \gamma_{-1})$ fits in the theory of Carbonaro, Mauceri and Meda~\cite{CMM}. The balls of radius not bigger than 1 with respect to $\rho$, called \emph{admissible}, are (equivalent to) the classical ``hyperbolic'' balls (see~\cite{CMMfinita}). Thus, we obtain a Hardy-type space $H^1(\gamma_{-1})$ as defined in~\cite{CMM} (see also Definition~\ref{def:H1} below) of functions in $L^1(\gamma_{-1})$ that admit a decomposition in terms of classical atoms supported only on admissible balls.

As we shall see, neither the imaginary powers $\As^{iu}$ nor the Riesz transform $\nabla \As^{-1/2}$ are bounded from $H^1(\gamma_{-1})$ to $L^1(\gamma_{-1})$. This is not completely surprising, for it happens in other contexts (e.g.~\cite{MMVSymmetric}). Thus, we consider a smaller Hardy-type space $X^1(\gamma_{-1})$ in the spirit of Mauceri, Meda and Vallarino~\cite{MMVHardy}, defined by $H^1(\gamma_{-1})$-atoms satisfying an additional cancellation condition. Roughly speaking, an $X^1(\gamma_{-1})$-atom is an $H^1(\gamma_{-1})$-atom which is orthogonal to all functions whose image under $\As$ is constant on admissible balls. It turns out that from the space $X^1(\gamma_{-1})$ so defined the imaginary powers are bounded into $L^1(\gamma_{-1})$, but still Riesz transforms are not. This was unexpected, since both the space $X^1$ of~\cite{MMVHardy} adapted to the Laplace-Beltrami operator, and the space $X^1(\gamma)$ adapted to the Ornstein-Uhlenbeck operator and the Gauss measure~\cite{B} \emph{are enough} as endpoint spaces. In addition to this, this failure of boundedness is rather strong: for every $0\leq \lambda<1$, there exists an $X^1(\gamma_{-1})$-atom $a$ such that $\Rs_\lambda a$ is not in $L^1(\gamma_{-1})$.

\smallskip

The techniques introduced by Mauceri, Meda and Vallarino~\cite{MMV} hinge on the fact that the inverse of the Laplace-Beltrami operator preserves the support of $X^1$-atoms, and this is a consequence only of the geometry of the manifold and the definition of $X^1$-atoms. And this is the case also in our setting for $\As$ (see Theorem~\ref{support:pres}). However, if one changes operator into a \emph{translation} of $\As$, say $\As+\lambda I$, to maintain this property of preservation of the support the cancellation condition on the atoms must be changed accordingly. This brings us to the definition of an $X^1_\lambda(\gamma_{-1})$-atom, $\lambda \geq 0$, as an $H^1(\gamma_{-1})$-atom which is orthogonal to all functions whose image under the translation $\As+\lambda I$ is constant. These atoms give rise to new atomic Hardy spaces $X^1_\lambda(\gamma_{-1})$, one for each non-negative translation $\lambda$ of $\As$ (see Definition~\ref{def:X1mu} below).

In this paper, we characterize the boundedness of $(\As+\lambda I)^{iu}$ and $\Rs_\lambda$ from $H^1(\gamma_{-1})$ to $L^1(\gamma_{-1})$ and from $X^1_\mu(\gamma_{-1})$ to $L^1(\gamma_{-1})$, for every $\lambda, \mu \geq0$. In particular, we prove that

\begin{itemize}
\item \emph{the imaginary powers $(\As+\lambda I)^{iu}$ are bounded from $H^1(\gamma_{-1})$ to $L^1(\gamma_{-1})$ if and only if $\lambda>0$. They are bounded from $X^1_\mu(\gamma_{-1})$ to $L^1(\gamma_{-1})$ if and only if either $\lambda=\mu=0$ or $\lambda>0$};
\item \emph{the Riesz transforms $\Rs_\lambda$ are bounded from $H^1(\gamma_{-1})$ to $L^1(\gamma_{-1})$ if and only if $n=1$ and $\lambda>1$. They are bounded from $X^1_\mu(\gamma_{-1})$ to $L^1(\gamma_{-1})$ if and only if either $\lambda=\mu=1$ or $\lambda>1$}.
\end{itemize}
These results are the content of Theorems~\ref{ImpowH1}, \ref{RieszH1}, \ref{teo:X1Impow} and \ref{teo:X1Riesz}, and are summarized in the following tables.
\vspace{-0.1cm}
\renewcommand{\arraystretch}{1.4}
\begin{table}[h]
\centering
\begin{tabular}{|c|c|c|}
\hline
$(\As +\lambda I)^{iu}\colon H^1(\gamma_{-1})\to L^1(\gamma_{-1})$ & \multicolumn{2}{c|}{$ (\As +\lambda I)^{iu}\colon X^1_\mu(\gamma_{-1})\to L^1(\gamma_{-1})$} \\ \hline
\multirow{2}{*}{bounded iff $\lambda >0$}                                 & $\lambda=0$                                       & $\lambda>0$                                              \\ \cline{2-3} 
                                                                                  & bounded iff $\mu=0$                               & bounded $\forall \mu\geq 0$                              \\ \hline
\end{tabular}
\end{table}
\vspace{-0.4cm}
\renewcommand{\arraystretch}{1.5}
\begin{table}[h]

\centering
\begin{tabular}{|c|c|c|c|}
\hline
$\Rs_\lambda \colon H^1(\gamma_{-1})\to L^1(\gamma_{-1}) $                                   & \multicolumn{3}{c|}{$\Rs_\lambda\colon X^1_\mu(\gamma_{-1})\to L^1 (\gamma_{-1})$} \\ \hline
\multirow{2}{*}{\begin{tabular}[c]{@{}c@{}}bounded iff \\ $n=1$ and $\lambda>1$\end{tabular}} & $0\leq \lambda <1$             & $\lambda=1$          & $\lambda\textgreater1$         \\ \cline{2-4} 
                                                                                              & unbounded $\forall \mu\geq 0$  & bounded iff $\mu=1$  & bounded $\forall \mu\geq 0$  \\ \hline
\end{tabular}
\end{table}

It is interesting to notice that both for the imaginary powers and for the Riesz transforms there is a ``critical translation'' (the null translation and the translation equal to $1$, respectively) for which the boundedness $X^1_\mu(\gamma_{-1})\to L^1(\gamma_{-1})$ holds if and only if the translation $\mu$ associated with the atomic space coincides with the translation of the operator. Beyond this threshold, the translation associated with the atomic space plays instead no role. As a consequence of this, moreover, the spaces $X^1_0(\gamma_{-1})$ and $X^1_1(\gamma_{-1})$ turn out to be different from each other and from any other $X^1_\mu(\gamma_{-1})$, $\mu \neq 0,1$. At the moment, we are not able to say the same for every non-negative translation (but see Corollary~\ref{corlnotm}, (2) below). We also observe that the critical translation for the Riesz transforms is exactly the value (or more precisely, the lower bound in terms of the Euclidean metric tensor) of the Bakry-Emery curvature tensor of the weighted manifold $(\R^n,\gamma_{-1})$. This seems to be the analogue of the role played by this curvature on $L^p(\gamma_{-1})$, $1<p<\infty$ (see~\cite{Bakry, CarbDrag}).

The strategy we adopt to prove the boundedness results from $X^1_\mu(\gamma_{-1})$ to $ L^1(\gamma_{-1})$ is strongly related to that of Mauceri, Meda and Vallarino~\cite{MMV}. This is based on the crucial result for which the boundedness of an operator $T$ from their atomic space $X^1$ to $L^1$ is equivalent to the ``uniform boundedness'' of $T$ on $X^1$-atoms~\cite[Proposition 6.3]{MMV-Harmonic}. More precisely, they prove that if
\[
\sup \{ \|T a\|_1 \colon \, \mbox{$a$ is an $X^1$-atom}\}<\infty
\]
then $T$ is bounded from $X^1$ to $L^1$ (the converse is always true). This result is based on the isomorphy of the duals of $X^1$ and $X^1_{\fin}$, the space of \emph{finite} linear combinations of $X^1$-atoms. In our case, we do not have such a powerful result for we do not have a characterization of the dual space yet. Therefore, we do not know whether the uniform boundedness of an operator on $X^1_\mu(\gamma_{-1})$-atoms (in the sense above) is enough to conclude its boundedness from $X^1_\mu(\gamma_{-1})$ to $L^1(\gamma_{-1})$.

Nevertheless, by a simple and classical argument (see e.g.~\cite[p.\ 95]{Grafakos}), an operator which is both uniformly bounded on certain atoms \emph{and of weak type $(1,1)$}, extends to a bounded operator from the whole corresponding atomic space to $L^1$. Thus, the weak type $(1,1)$ of the translated Riesz transforms and the imaginary powers associated with $\As$ is not only interesting on its own, but turns out to be somewhat necessary to our proof.  For $\lambda \geq0$ and $u\in \R\setminus \{0\}$, we prove that
\begin{itemize}
\item \emph{the imaginary powers $(\As+\lambda I)^{iu}$ are of weak type $(1,1)$ for every $\lambda\geq 0$};
\item \emph{the Riesz transforms $\Rs_\lambda$ are of weak type $(1,1)$ for every $\lambda \geq 1$}.
\end{itemize}
This is the content of Theorem~\ref{teo:weaktype}.

Since the measure $\gamma_{-1}$ is locally doubling on admissible balls but it is not globally doubling, we follow the classical procedure of splitting the operators in their local and global parts, and prove the weak type $(1,1)$ of each part separately. By means of the local doubling condition, their local part can be traced back to the classical Calder\'on-Zygmund theory, which only involves certain estimates of the kernels of these operators and their first derivatives. As for the global parts, our strategy relies on a new proof of the weak type $(1,1)$ of the Mehler maximal operator of $\As$ (whose first proof is due to Salogni~\cite{Salogni}) which is inspired by a new proof of the weak type $(1,1)$ of the Mehler maximal operator of the Ornstein-Uhlenbeck operator given by the author in~\cite{B}. By this, we can prove that in the global region a kernel of an operator of weak type $(1,1)$, related to the Mehler maximal kernel, controls both those of the imaginary powers and those of the Riesz transforms associated with a translation not smaller than 1.

\subsection*{Structure of the paper and notation}
In Section~\ref{Sec:prel} we prove some preliminary results which will be used throughout the paper, and obtain the Schwartz kernels of the operators $(\As+\lambda I)^{iu}$ and $\Rs_\lambda$. In Section~\ref{Sec:HardyOUrovesciato} we define the atomic Hardy spaces $H^1(\gamma_{-1})$ and $X^1_\lambda(\gamma_{-1})$, $\lambda\geq 0$, and investigate some properties which will be of use thereafter. In particular, Subsection~\ref{subsec:Ex} contains some fundamental classes of functions. In Section~\ref{Sec:weakOUrovesciato} we prove the weak type $(1,1)$ of the imaginary powers and the Riesz transforms as explained above. In Section~\ref{BoundednessH1OUrovesciato}, we prove the mentioned boundedness results from $H^1(\gamma_{-1})$ to $L^1(\gamma_{-1})$, while the boundedness results from $X^1_\mu(\gamma_{-1})$ to $L^1(\gamma_{-1})$ are proved in Section~\ref{BoundednessX1OUrovesciato}.

\smallskip

If $\nu$ is a positive measure on $\R^n$ and $1\leq p<\infty$, we denote by $L^p(\nu)$ the space of (equivalence classes of) measurable functions $f$ on $\R^n$ such that $|f|^p$ is integrable with respect to $\nu$, with the usual norm which will be denoted by $\|f\|_{L^p(\nu)}$. The space $L^\infty(\nu)$ will be the space of measurable functions which are essentially bounded with respect to $\nu$. If $p=2$, $L^2(\nu)$ is a Hilbert space and its scalar product will be denoted by $( \cdot\,, \cdot)_{L^2(\nu)}$. When there is no risk of confusion, we will write simply $L^p$, $\|f\|_p$ and $( \cdot\, ,\, \cdot)$.

The Lebesgue measure will be denoted by $\dd x$, and with a slight abuse of notation we will denote with $\gamma_{-1}$ both the function $x\mapsto \pi^{n/2} e^{|x|^2}$ on $\R^n$ and the measure with density $\gamma_{-1}$ with respect to $\dd x$. The measure of a set $E\subset \R^n$ with respect to $\dd x$ and $\gamma_{-1}$ will be denoted by $|E|$ and $\gamma_{-1}(E)$ respectively. Given a bounded operator $T$ on $L^2(\gamma_{-1})$, we denote by $K_T$ the Schwartz kernel of $T$ and by $k_T$ the kernel of $T$ with respect to the measure $\gamma_{-1}$. In other words
\begin{equation}\label{rulekernel}
Tf(x)= \int_{\R^n} K_T(x,y) f(y)\, \dd y = \int_{\R^n} k_T(x,y) f(y)\gamma_{-1}(y)\, \dd y.
\end{equation}
In the paper we shall use \emph{Euclidean} balls. All throughout the paper, we shall use the letters $c$ and $C$ to denote constants, not necessarily equal at different occurrences. For any quantity $A$ and $B$, we write $A\lesssim B$ by meaning that there exists a constant $c>0$ such that $A\leq c \,B$. If $A\lesssim B$ and $B\lesssim A$, we write $A\approx B$ .

\section{Preliminaries}\label{Sec:prel}
Denote with $\gamma$ both the Gauss function $\gamma\coloneqq 1/\gamma_{-1}$ and the Gauss measure whose density with respect to $\dd x$ is $\gamma$. Let $\Ls$ be the Ornstein-Uhlenbeck~operator, i.e.\ the closure on $L^2(\gamma)$ of the operator
\[\mathcal{L}_0f(x)\coloneqq -\frac{1}{2}\Delta f(x) + x\cdot \nabla f(x), \qquad f\in C_c^\infty.\]
It is well-known that $\Ls$ is self-adjoint. We recall that its $L^2(\gamma)$-spectrum is the set of nonnegative integers $\{0,1,\dots\}$. We refer the reader to~\cite{Sjogren} for a detailed introduction. 

Let $U:L^2(\gamma)\to L^2(\gamma_{-1})$ be the isometry
\[Uf=f\gamma \qquad \forall f\in L^2(\gamma).\]
An easy computation shows that for every $f\in C_c^\infty(\R^n)$
\begin{equation}\label{equivOU}
\As f= U(\Ls +nI)U^{-1}f
\end{equation}
and since the isometry $U$ preserves the space of test functions, the operators $\Ls+n I$ and $\As$ are unitarily equivalent (see~\cite{Salogni}). Therefore, $\As$ has spectral gap equal to $n$, and its $L^2(\gamma_{-1})$-spectrum is the set of positive integers $\{n,\, n+1,\, n+2,\, \dots\}$. 

From the unitary equivalence~\eqref{equivOU}, it is possible to obtain the Mehler kernel of the semigroup $(e^{-t\mathcal{A}})_{t>0}$ from that of $e^{-t\Ls}$. This was performed in~\cite{Salogni}, where it is proved that the Mehler kernel with respect to the measure $\gamma_{-1}$ is
\begin{equation}\label{ht}
h_t(x,y)= \frac{e^{-nt}}{\pi^n(1-e^{-2t})^{n/2}} \exp \bigg( -\frac{|x+y|^2}{2(1+e^{-t})} - \frac{|x-y|^2}{2(1-e^{-t})}\bigg).
\end{equation}
In other words,
\[e^{-t\mathcal{A}} f(x) = \int_{\R^n} h_t(x,y)f(y)\gamma_{-1}(y)\, \dd y\]
for every $t>0$. If we denote with $H_t(x,y)\coloneqq h_t (x,y) \gamma_{-1}(y)$ the Mehler kernel of $e^{-t\mathcal{A}}$ with respect to the Lebesgue measure,
\[H_t(x,y) = \frac{e^{-nt}}{\pi^{n/2}(1-e^{-2t})^{n/2}} e^{-\frac{|x-e^{-t}y|^2}{1-e^{-2t}}}.\]
For $1\leq p<\infty$, denote by $\As_p$ the operator $\As$ (in the distributional sense) on $\Dom(\As_p)\coloneqq \{f\in L^p(\gamma_{-1})\colon\, \As f\in L^p(\gamma_{-1})\}$, and by $\sigma_p(\As)$ its $L^p(\gamma_{-1})$-spectrum. Then, we have the following result.

\begin{proposition}\label{propspectrum}
For every $1<p<\infty$, $\sigma_p(\As)=\{n,\, n+1,\, n+2,\, \dots\}$, while $\sigma_1(\As)=\{z\in \C \colon \, \Re z \geq 0\}$.
\end{proposition}
\begin{proof}
Since $e^{-t\As}$ is Markovian (cf.~\cite{Salogni}), $\sigma_p(\As)= \sigma_2(\As)$ for every $1<p<\infty$ by~\cite[Theorem 1.6.3]{Davies1}. If $p=1$, consider the isometry $V\colon L^1(\gamma_{-1}) \to L^1(\dd x)$ given by $Vf=f\gamma_{-1}$. Then
\[V\As_1 V^{-1} \colon L^1(\dd x)\to L^1(\dd x),\quad V\As_1 V^{-1} = \mathcal{L}_1 +nI,\]
where $\mathcal{L}_1$ is the Ornstein-Uhlenbeck operator on $L^1(\dd x)$. By~\cite[Theorem 4.12]{Metafune}, the $L^1(\dd x)$-spectrum of $\mathcal{L}_1$ is the half-plane $\{z\in \C \colon \, \Re z \geq - n\}$. Therefore the $L^1(\dd x)$-spectrum of $\mathcal{L}_1+nI$ is the half-plane $\{z\in \C \colon \, \Re z \geq 0\}$, and thus this is the $L^1(\gamma_{-1})$-spectrum of $\As_1$.
\end{proof}

\subsection{Kernels of integral operators}
Let $\lambda \geq 0$, and observe that \[\sigma_2(\As+\lambda I)= \{n+\lambda,\, n+\lambda +1,\, \dots \}\]
by Proposition~\ref{propspectrum}. Denote by $(\Pc_k)$, $k=n,\, n+1,\, \dots$ the spectral resolution of $\As$. For every $z\in \C$, define the operator $(\As+\lambda I)^z$ via the spectral theorem on $L^2(\gamma_{-1})$ as
\begin{multline*}
(\As+\lambda I)^z= \sum_{k=n}^\infty (k+\lambda)^z \mathcal{P}_k,\\ \Dom((\As+\lambda I)^z)=\Big\{f\in L^2(\gamma_{-1})\colon \sum_{k=n}^\infty (k+\lambda)^{2 \Re z} \|\mathcal{P}_k f\|_{L^2(\gamma_{-1})}^2<\infty\Big\}.
\end{multline*}
If $\Re z \leq 0$, $(\As+\lambda I)^z$ is bounded on $L^2(\gamma_{-1})$ and $\Dom((\As+\lambda I)^z)= L^2(\gamma_{-1})$. If $\Re z \geq 0$, observe that $C_c^\infty \subset \Dom ((\As+\lambda I)^z)$ since by the spectral theorem $(\As+\lambda I)^z=(\As+\lambda I)^{z-N} \,(\As+\lambda I)^N $, where $N=[\Re\,z]+1$.\par

For every $x \neq y$ define
\[
K_z^\lambda(x,y)= \frac{1}{\Gamma(-z)} \int_0^\infty e^{-\lambda t} t^{-z-1} H_t(x,y)\, \dd t.
\]
\begin{remark}
If $z\in \N$, $1/\Gamma(-z)=0$ and hence $K_z^\lambda(x,y)= 0$ for every $x\neq y$.
\end{remark}
For $s>0$, let
\begin{equation}\label{Ns_prima}
N_s \coloneqq \{(x,y)\in \R^n\times \R^n \colon |x-y|\leq s/(1+|x|+|y|) \}.
\end{equation}

\begin{lemma}\label{lemmaKz}
Let $\lambda \geq 0$, $z\in \C$, and $s>0$. For every $(x,y)\in N_s$, $x\neq y$,
\[
\int_0^\infty e^{-\lambda t} t^{-\Re z -1} H_t(x,y)\, \dd t\leq C(s,z,\lambda)\frac{1}{|x-y|^{n+2 \Re z}}.
\]
In particular, for every $x\neq y$ the integral defining $K_z^\lambda(x,y)$ is absolutely convergent.
\end{lemma}
\begin{proof} Let $\lambda \geq 0$, $z\in \C$, $s>0$ and $(x,y)\in N_s$, $x\neq y$. Then 
\begin{align*}
|x-e^{-t}y|^2 \geq |x-y|^2 -2(1-e^{-t})|y||x-y|\geq |x-y|^2 -2s(1-e^{-t}),
\end{align*}
which yields
\[
\int_0^\infty e^{-\lambda t} t^{-\Re z -1} H_t(x,y)\, \dd t \leq C e^{cs}\int_0^\infty t^{-\Re z -1} e^{-(n+\lambda )t}\frac{1}{(1-e^{-2t})^{n/2}} e^{-\frac{|x-y|^2}{1-e^{-2t}}}\, \dd t.
\]
We split the latter integral into the sum of the integrals on $(0,1)$ and $(1,\infty)$. Then
\begin{align*}
\int_0^1 t^{-\Re z -1} e^{-(n+\lambda)t}\frac{1}{(1-e^{-2t})^{n/2}} e^{-\frac{|x-y|^2}{1-e^{-2t}}}\, \dd t& \leq C \int_0^1 t^{-\Re z -1 -\frac{n}{2}} e^{-c\frac{|x-y|^2}{t}}\, \dd t \\& \leq C(z) |x-y|^{-n- 2 \Re z}
\end{align*}
while
\[
\int_1^\infty t^{-\Re z -1} e^{-(n+\lambda) t}\frac{1}{(1-e^{-2t})^{n/2}} e^{-\frac{|x-y|^2}{1-e^{-2t}}}\, \dd t \leq C e^{-|x-y|^2} \int_1^\infty t^{-\Re z -1} e^{-(n+\lambda)t} \, \dd t
\]
from which the stated estimate follows.
\end{proof}
The following proposition shows that $K_{-z}^\lambda$ is the kernel of $(\As+\lambda I)^{-z}$ if $\Re\, z>0$. 
\begin{proposition}\label{Prop:negativeline}
Let $\lambda \geq 0$. For every $z\in\C$ with $\Re\, z>0$
\[
(\As+\lambda I)^{-z}=\frac{1}{\Gamma(z)}\int_0^\infty e^{-\lambda t}t^{z-1} e^{-t\As} \, \dd t
\]
where the integral converges in the weak operator topology of $L^2(\gamma_{-1})$. Moreover, 
\[
(\As+\lambda I)^{-z}f(x)= \int_{\R^n} K_{-z}^\lambda(x,y)f(y)\, \dd y
\]
for every $f\in C^\infty_c(\R^n)$ and for almost every $x\in \R^n$.
\end{proposition}
\begin{proof}
Let $f,g\in L^2(\gamma_{-1})$. Then
\begin{align*}
((\As+\lambda I)^{-z}f,g)&= \sum_{k=n}^\infty (k+\lambda)^{-z}(\Pc_k f,g)\\
&= \frac{1}{\Gamma(z)}\sum_{k=n}^\infty \int_0^\infty t^{z-1} e^{-(k+\lambda)t}\dd t\, (\Pc_k f,g)= \frac{1}{\Gamma(z)}\int_0^\infty e^{-\lambda t} t^{z-1}( e^{-t\As}f,g)\, \dd t
\end{align*}
where one can interchange the order of summation and integration since $\Re\, z>0$ and
\begin{align*}\label{PlanCauc}
\sum_{k=n}^\infty |(\Pc_k f,g)| &=\sum_{k=n}^\infty |(\Pc_k f,\Pc_k g)| \leq \| f\|_2\|g\|_2.
\end{align*}
This proves the first assertion of the statement. As for the second assertion, if $f,g\in C_c^\infty$ we get
\begin{align*}
((\As+\lambda I)^{-z}f,g)&= \frac{1}{\Gamma(z)}
\int_0^\infty e^{-\lambda t} t^{z-1}\int_{\R^n} \int_{\R^n}H_t(x,y) f(y)\, \dd y \, \bar{g}(x)\dd \gamma_{-1}(x)\, \dd t.
\end{align*}
Since when $\Re z>0$ the function $y\mapsto K_{-z}^\lambda (x,y)$ is locally integrable for every $x\in \R^n$ by Lemma~\ref{lemmaKz}, this integral is absolutely convergent. Thus, the conclusion follows by Fubini's theorem.
\end{proof}
Let now $z\in \C$. In the following theorem we show that, outside the diagonal in $\R^n\times\R^n$, $K_z^\lambda(x,y)$ is still the Schwartz kernel of $(\As+\lambda I)^z$ with respect to the Lebesgue measure.
\begin{theorem}\label{teo:complex_powers}
Let $\lambda \geq 0$ and $z\in \C$. Then, for every $f\in C_c^\infty$
\[
(\As+\lambda I)^z f(x) = \int_{\R^n} K_z^\lambda(x,y) f(y)\, \dd y
\]
for all $x$ outside the support of $f$.
\end{theorem}

\begin{proof}
Fix $f,g\in C_c^\infty$ with disjoint support. For $\lambda \geq 0$ and $z\in \C$ define the functions
\[
F_\lambda(z)= ((\As+\lambda I)^z f,g)
\]
and 
\[
G_\lambda(z)= \int_{\R^n}\int_{\R^n} K_z^\lambda(x,y) f(y)\bar{g}(x)\, \dd y \, \dd \gamma_{-1}(x).
\]
We prove that both functions are holomorphic on $\C$.

Since $f$ and $g$ have disjoint support, the integral defining $G_\lambda(z)$ converges absolutely for every $z\in \C$ by Lemma~\ref{lemmaKz}. Since the function $z\mapsto t^{z-1}$ is holomorphic for every $t>0$, $G_\lambda$ is holomorphic by Fubini's, Goursat's and Morera's theorems.

As for $F_\lambda$, if $N= [\Re z] +1$ then
\begin{align*}\label{serieF}
F_\lambda(z) &= ((\As+\lambda I)^{z-N}(\As+\lambda I)^N f,g)\\&= \sum_{k=n}^\infty (k+\lambda)^{z-N} (\Pc_k (\As+\lambda I)^N f,g)= \sum_{k=n}^\infty (k+\lambda)^{z-N} (\Pc_k (\As+\lambda I)^N f,\Pc_k g)
\end{align*}
and the series converges uniformly on compact sets, since
\[
\sum_{k=n}^\infty |(k+\lambda)^{z-N} (\Pc_k (\As+\lambda I)^N f,\Pc_k g)| \leq (n+\lambda)^{\Re z -N} \|(\As+\lambda I)^N f\|_2 \|g\|_2.
\]
Thus, $F_\lambda$ is holomorphic in $\C$. Since $F_\lambda$, $G_\lambda$ are both holomorphic in $\C$ and they coincide for $\Re\, z<0$ by Proposition~\ref{Prop:negativeline}, the statement follows by the uniqueness of the analytic continuation.
\end{proof}
By means of Proposition~\ref{Prop:negativeline} and Theorem~\ref{teo:complex_powers}, we may obtain the kernels of the imaginary powers and the Riesz transforms associated with $\As$. To be more concise, we shall often adopt the notation of~\cite{MMS1}
\[\phi(r,x,y)\coloneqq \frac{ry -x}{\sqrt{1-r^2}},\qquad \psi(r,x,y)\coloneqq \frac{rx-y}{\sqrt{1-r^2}}
\]
for $r\in (0,1)$. Observe that $|\phi|^2 -|\psi|^2 = |x|^2 -|y|^2$ for every $x,y\in \R^n$. We recall that we adopt the notational convention~\eqref{rulekernel}.

By Theorem~\ref{teo:complex_powers}, the kernel of $(\As+\lambda I)^{iu}$, $u\neq 0$, with respect to the Lebesgue measure is
\begin{equation}\label{impowLeb}
\begin{split}
K_{(\As+\lambda I)^{iu}}(x,y) &= \frac{c(u)}{\pi^{n/2}} \int_0^\infty t^{-iu-1} \frac{e^{-(n+\lambda)t}}{(1-e^{-2t})^{n/2}} e^{-|\phi(e^{-t},x,y)|^2}\, \dd t\\&= \frac{c(u)}{\pi^{n/2}} \int_0^1 \frac{r^{n+\lambda-1} (-\log r)^{-iu-1}}{(1-r^{2})^{n/2}} e^{-|\phi(r,x,y)|^2}\, \dd r
\end{split}
\end{equation}
where we used the change of variables $e^{-t}=r$, and $c(u)= \Gamma(-iu)^{-1}$. The kernel of $(\As+\lambda I)^{iu}$ with respect to the measure $\gamma_{-1}$ is instead
\begin{equation}\label{impowkernel}
\begin{split}
k_{(\As+\lambda I)^{iu}}(x,y) &= \frac{c(u)}{\pi^n} e^{-|x|^2}\int_0^1 \frac{r^{n+\lambda-1} (-\log r)^{-iu-1}}{(1-r^{2})^{n/2}} e^{-|\psi(r,x,y)|^2} \, \dd r.
\end{split}
\end{equation}
By Proposition~\ref{Prop:negativeline}, the kernel of $(\As+\lambda I)^{-1/2}$, $\lambda\geq 0$ with respect to the Lebesgue measure is
\[
K_{(\As+\lambda I)^{-1/2}}(x,y) = \frac{1}{\pi^{(n+1)/2}}\int_0^\infty t^{-1/2} \frac{e^{-(n+\lambda) t}}{(1-e^{-2t})^{n/2}}e^{-|\phi(e^{-t},x,y)|^2} \, \dd t.\]
By differentiation, we obtain the kernel of the Riesz transforms $\Rs_\lambda= \nabla(\mathcal{A}+\lambda I)^{-1/2}$ with respect to the Lebesgue measure,
\begin{equation}\label{tRieszkernel}
\begin{split}
K_{\Rs_\lambda}(x,y)&= -2\pi^{-\frac{n+1}{2}} \int_0^\infty t^{-1/2} \frac{e^{-(n+\lambda)t}}{(1-e^{-2t})^{(n+2)/2}}(x-e^{-t}y)e^{-|\phi(e^{-t},x,y)|^2} \, \dd t
\\&= -2 \pi^{-\frac{n+1}{2}}\int_0^1 \frac{r^{n+\lambda-1}}{(1-r^2)^{(n+2)/2}\sqrt{-\log r}} (x -ry) e^{-|\phi(r,x,y)|^2}\, \dd r 
\end{split}
\end{equation}
where we used again the change of variables $e^{-t}=r$. Their kernels with respect to the measure $\gamma_{-1}$ will be
\begin{align}\label{Rieszkernel}
k_{\Rs_\lambda}(x,y) = -2\pi^{-n-\frac{1}{2}} e^{-|x|^2}\int_0^1 \frac{r^{n+\lambda-1}}{(1-r^2)^{(n+2)/2}\sqrt{-\log r}} (x-ry) e^{-|\psi(r,x,y)|^2} \, \dd r.
\end{align}
We shall denote with $(\Rs_\lambda)_j$, $j=1,\dots, n$, the $j^{th}$ component of the vector-valued operator $\Rs_\lambda$. If $X$ and $Y$ are Banach spaces, with a slight abuse we say that $\Rs_\lambda$ is bounded from  $X$ to $Y$ by meaning that $(\Rs_\lambda)_j$ is bounded from $X$ to $Y$ for every $j=1,\dots,n$. 

\begin{remark}\label{RemLp}
Let $p\in (1,\infty)$ and $u\neq 0$. By a result of Salogni~\cite[Theorem 3.4.3]{Salogni} the imaginary powers $\As^{iu}$ associated with $\As$ are bounded on $L^p(\gamma_{-1})$, and the same holds for the shifted Riesz transforms $\nabla (\As + \lambda I)^{-1/2}$ for every $\lambda \geq 1$, by a celebrated theorem of Bakry~\cite{Bakry}. Then, since
\[
\nabla\As^{-1/2}= \nabla (\As+I)^{-1/2} (\As + I)^{1/2}\As^{-1/2}
\]
and since the operator $(\As + I)^{1/2}\As^{-1/2}$ is bounded on $L^p(\gamma_{-1})$ by~\cite[Theorem VII.9.4]{DunfordSchwartz}, we obtain also the boundedness of the Riesz transforms $\nabla \As^{-1/2}$ on $L^p(\gamma_{-1})$. Nevertheless, as a consequence of our Theorems~\ref{ImpowH1} and~\ref{RieszH1}, both $\As^{iu}$ and $\nabla \As^{-1/2}$ are unbounded on $L^1(\gamma_{-1})$. 
\end{remark}

\section{Hardy Spaces}\label{Sec:HardyOUrovesciato}
As already pointed out in the introduction, the atoms of our atomic spaces are classical atoms supported in (dilations of) ``hyperbolic'' balls, as defined below. For every ball $B$ we write $c_B$ to denote its center, $r_B$ for its radius, and $kB$ to denote the ball with same center $c_B$ and radius $k\, r_B$.

\begin{definition}\label{admissible-ball}
Given $s>0$, we call \emph{admissible ball at scale $s$} a ball $B$ of centre $c_B$ and radius $r_B \leq s \min (1,1/|c_B|)$. The family of all admissible balls at scale $s$ will be denoted by $\mathcal{B}_s$. Balls in $\mathcal{B}_1$ will often be referred to only as \emph{admissible balls}.
\end{definition}
\begin{lemma}\label{adm}
Let $s>0$. There exist some constants $c_1(s)$, $c_2(s)$ such that for every $B\in \mathcal{B}_s$ and every subset $E\subseteq B$ 
\[
c_1(s) e^{|c_B|^2} |E| \leq \gamma_{-1}(E) \leq c_2(s) e^{|c_B|^2} |E|.\]
In particular, $\gamma_{-1}(B) \approx e^{|c_B|^2}|B|$ for every $B\in \mathcal{B}_1$ and $\gamma_{-1}$ is locally doubling on admissible balls.
\end{lemma}

\begin{proof}
Let $B\in \mathcal{B}_s$. For every $y \in B$, then, $|y - c_B|\leq r_B \leq s \min (1,1/\abs{c_B})$. Therefore
\[|y|^2 \leq (|y -c_B| + |c_B|)^2 \leq r_B^2 +|c_B|^2 +2r_B |c_B| \leq |c_B|^2 +s^2 +2s,\]
while
\[|y|^2 \geq (|y -c_B| - |c_B|)^2 = |c_B|^2 +|y-c_B|^2 -2|c_B||y-c_B| \geq |c_B|^2 -2s.\]
The statement is now easy to verify.
\end{proof}

\subsection{The Hardy space $H^1(\gamma_{-1})$}\label{Sec:H1}
In this section, we recall the definition of Carbonaro, Mauceri and Meda's atomic Hardy space~\cite{CMM} in the case when the metric measure space is $(\R^n, \rho, \gamma_{-1})$.
\begin{definition}\label{atom}
An \emph{$H^1$-atom} is a function $a$ supported in a ball $B\in \mathcal{B}_1$ such that
\begin{itemize}
\item[\rmi] $\|a\|_{L^2(\gamma_{-1})} \leq \gamma_{-1}(B)^{-1/2}$,
\item[\rmii] $\int_{\R^n} a(x)\, \gamma_{-1}(x)\, \dd x=0$.
\end{itemize}
\end{definition}

\begin{definition}[cf.~\cite{CMM}]\label{def:H1}
We define the Hardy space
\begin{align*}
H^1(\gamma_{-1})\coloneqq\big\{ f \in L^1(\gamma_{-1}) \colon \mbox{$f= \sum_j c_j a_j$},\; \mbox{$a_j$ is an $H^1$-atom}, \; (c_j) \in \ell^1\big\}
\end{align*}
with the norm
\[
\|f\|_{H^1(\gamma_{-1})}\coloneqq \inf \{\| (c_j)\|_{\ell^1}\colon\mbox{$f= \sum_j c_j a_j$}, \mbox{ $a_j$ $H^1$-atom} \}.
\]
\end{definition}

\begin{remark}\label{rem:cubic}
If in Definition~\ref{atom} we replace atoms supported in balls at scale $1$ with atoms supported in balls at any fixed scale $s>0$, we obtain the same space $H^1(\gamma_{-1})$ with an equivalent norm. The same holds if we replace the $L^2$-size condition in Definition~\ref{atom}, (i) with any other $L^p$-size condition, $p\in (1,\infty]$. This is indeed a consequence of~\cite[Proposition 4.3, (ii) and Theorem 6.1]{CMM}, and of an \emph{almost verbatim} repetition of the proof of~\cite[Theorem 2.2]{MMS1}, together with Lemma~\ref{adm} above. In the notation of~\cite{CMM},
\[
H^{1,r}_s(\gamma_{-1})=H^{1,p}_{t}(\gamma_{-1}) \quad \forall \, r,p\in (1,\infty], \; s,t>0
\]
with equivalence of norms.\par
Therefore, since every admissible ball at scale $s$ is contained in a cube $Q$ of centre $c_Q$ and sidelength at most $2s\min (1,1/|c_Q|)$ (which we may call \emph{admissible cube at scale $s$}), while every such cube is contained in an admissible ball at scale $s\sqrt{n}$, we obtain the same space $H^1(\gamma_{-1})$ by considering atoms supported in admissible cubes at any fixed positive scale, with equivalence of norms.
\end{remark}

 \subsection{The Hardy spaces $X^1_\lambda(\gamma_{-1})$}\label{Sec:X1}
In this section, we introduce the new atomic Hardy spaces $X^1_\lambda(\gamma_{-1})$, $\lambda\geq0$, associated with different translations of $\As$. The reader should compare our definitions with those of~\cite{MMVAtomic}.

\begin{definition}\label{quasi-harmonic-lambda}
Let $\lambda \geq 0$, $\Omega$ be a bounded open set and $K$ be a compact set.
\begin{itemize}
\item We denote by $q^2_\lambda(\Omega)$ the space of all functions $u\in L^2(\Omega)$ such that $(\mathcal{A}+\lambda I)u$ is constant on $\Omega$, and by $q^2_\lambda(K)$ the space of functions on $K$ which are the restriction to $K$ of a function in $q^2(\Omega')$ for some bounded open $\Omega'\supset K$; 
\item we denote by $h^2_\lambda(\Omega)$ the space of all functions $u\in L^2(\Omega)$ such that $(\mathcal{A}+\lambda I)u=0$ on $\Omega$, and by $h^2(K)$ the space of functions on $K$ which are the restriction to $K$ of a function in $h^2_\lambda(\Omega')$ for some bounded open $\Omega'\supset K$.
\end{itemize}
The spaces $h^2_\lambda(\Omega)^\perp$ and $q^2_\lambda(\Omega)^\perp$ are the orthogonal complements of $h^2_\lambda(\Omega)$ and $q^{2}_\lambda(\Omega)$ in $L^2(\Omega,\gamma_{-1})$, respectively. The spaces $h^2_\lambda(K)^\perp$ and $q^2_\lambda(K)^\perp$ will be the orthogonal complements in $L^2(K,\gamma_{-1})$.
\end{definition}

\begin{definition}\label{atom-lambda}
Let $\lambda\geq0$. An \emph{$X^1_\lambda$-atom} is a function $a\in L^2(\gamma_{-1})$, supported in a ball $B\in \mathcal{B}_1$, such that
\begin{itemize}
\item[\rmi] $\|a\|_{L^2(\gamma_{-1})} \leq \gamma_{-1}(B)^{-1/2}$,
\item[\rmii] $a\in q^2_\lambda(\bar{B})^\perp$.
\end{itemize}
\end{definition}
\begin{definition}\label{def:X1mu}
For every $\lambda\geq 0$, the Hardy space $X^{1}_{\lambda}(\gamma_{-1})$ is the space
\begin{equation*}
X^{1}_{\lambda}(\gamma_{-1})\coloneqq\big\{ f \in L^1(\gamma_{-1}) \colon \mbox{$f= \sum_j c_j a_j$}, \mbox{ $a_j$ $X^1_\lambda$-atom }, (c_j) \in \ell^1\big\}
\end{equation*}
endowed with the norm
\[
\|f\|_{X^1_\lambda(\gamma_{-1})}\coloneqq \inf \{\| (c_j)\|_{\ell^1}\colon\mbox{$f= \sum_j c_j a_j$}, \mbox{ $a_j$ $X^1_\lambda$-atom} \}.
\]
\end{definition}
If $B\in \mathcal{B}_1$, the functions in $q^2_\lambda(\bar{B})$ will be referred to as $\lambda$-\emph{quasi-harmonic functions} on $B$.

\subsection{Support preservation of $(\As+\lambda I)^{-1}$ on atoms}\label{subSec:ECC}
The following result may be obtained by a straightforward adaptation of~\cite[Subsection 2.1]{B} to the current setting of the inverse Gauss measure, and its proof is omitted.
\begin{theorem}\label{support:pres}
Let $\lambda\geq 0$. For every $X^1_\lambda$-atom $a$ supported in an admissible ball $B$, $\supp (\As+\lambda I)^{-1}a \subseteq \bar{B}$ and
\[
\|(\As+\lambda I)^{-1} a\|_{L^2(\gamma_{-1})}\leq r_B^2\, \gamma_{-1} (B)^{-1/2}. \]
\end{theorem}

\subsection{Two important classes of functions}\label{subsec:Ex} 
In this section we introduce two families of functions that play an important role in the proof of the unboundedness results. One of the two also provides examples of non-trivial $\lambda$-harmonic functions.
\begin{definition}
Let $\lambda\geq 0$. For $\sigma, y\in \R^n$, define the functions
\begin{equation}\label{nonH1}
\Psi_{\lambda,\sigma}(y)= \int_0^\infty e^{-t} t^{(n+\lambda -2)/2} e^{2(\sigma,y)\sqrt{t}-|y|^2}\, \dd t
\end{equation}
and
\begin{equation}\label{nonX1}
\Phi_{\lambda,\sigma}(y)=\int_0^{\infty} e^{-t}t^{(n+\lambda-2)/2} \log(1/t)e^{2(\sigma,y) \sqrt{t}-|y|^2} \, \dd t.
\end{equation}
\end{definition}
\begin{lemma}\label{Psi}
Let $\lambda\geq 0$ and $B$ a ball in $\R^n$. Then for every $\sigma\in S^{n-1}$
\begin{itemize}
\item[\emph{(1)}] the function $\Psi_{\lambda, \sigma}$ is $\lambda$-harmonic in $\R^n$. In particular, $\Psi_{\lambda,\sigma}\in h^2_\lambda(B)$. Moreover $(\As+\lambda I)\Phi_{\lambda,\sigma}=2\Psi_{\lambda,\sigma}$;
\item[\emph{(2)}] $\Psi_{\lambda,\sigma}$ is not constant on any open subset of $\R^n$;
\item[\emph{(3)}] there exists a function $\psi$ in $C^\infty_c(B)$ with integral zero such that \break
$
(\psi,\Psi_{\lambda,\sigma})_{L^2(B,\gamma_{-1})}\not=0$.
\end{itemize} 
\end{lemma}
\begin{proof}
Statement (1) is a consequence of the equality
\begin{align}\label{eqcancell}
 e^{-t}t^{(n+\lambda-2)/2}(\As+\lambda I) (e^{2(\sigma,y) \sqrt{t}-|y|^2})= e^{-|y|^2} \frac{\dd}{\dd t} \bigg( 2e^{2(\sigma, y)\sqrt{t} -t} t^{(n+\lambda)/2}\bigg)
\end{align}
for every $y\in \R^n$ and $\sigma$ such that $|\sigma|=1$, and this yields
\[(\As+\lambda I) \Psi_{\lambda, \sigma} (y)= \int_0^\infty e^{-t}t^{(n+\lambda-2)/2} (\As+\lambda I)( e^{2(\sigma,y) \sqrt{t}-|y|^2}) \, \dd t = 0\]
for every such $y$ and $\sigma$.\footnote{Equivalently,~\eqref{eqcancell} can be formulated by saying that $f_\sigma^\lambda(y,t) = e^{-t} t^{(n+\lambda)/2}e^{2(\sigma,y)\sqrt{t}}$ is a solution of the equation
\[
(\Ls +(\lambda + n)I) u_\sigma = 2t \frac{\dd}{\dd t} u_\sigma.
\]}
Moreover, by~\eqref{eqcancell}
\begin{align*}\label{psiatomo}
(\As +\lambda I)\Phi_{\lambda,\sigma}(y) &= e^{-|y|^2} \int_0^\infty \log(1/t) \frac{\dd}{\dd t} \bigg( 2e^{2(\sigma, y)\sqrt{t} -t} t^{(n+\lambda)/2}\bigg)\, \dd t = 2 \Psi_{\lambda,\sigma}(y),
\end{align*}
where the last equality holds by integration by parts. 
As for (2), observe that $\Psi_{\lambda,\sigma}$ is the restriction to $\R^n$ of an entire function on $\C^n$. Thus, if it were constant on some open subset of $\R^n$, it would be constant everywhere. Since 
\[ \nabla \Psi_{\lambda,\sigma} (y) = -2y\Psi_{\lambda,\sigma}(y) +2\sigma \Psi_{\lambda+1,\sigma}(y)
\]
we have $|\nabla \Psi_{\lambda,\sigma} (0)| =2 \Psi_{\lambda+1,\sigma}(0) >0$. Thus, $\Psi_{\lambda,\sigma}$ is not constant in a neighbourhood of the origin.\par
To prove (3), denote by $C^\infty_{c,0}(B)$ the space of functions in
$C^\infty_c(B)$ with integral zero. Since the orthogonal of $C^\infty_{c,0}(B)$ in $L^2(B,\gamma_{-1})$ is the space of functions that are constant on $B$, (3) follows from (2).
\end{proof}

It is easily seen that both $\Psi_{\lambda,\sigma}$ and $\Phi_{\lambda,\sigma}$ are in $L^\infty_\loc$ for every $\lambda \geq 0$ and $\sigma \in S^{n-1}$, since they are smooth. In particular, they are in $L^2_\loc$. Thus, the integral $\int\Psi_{\lambda,\sigma} f \,\dd \gamma_{-1}$ is well defined for every $f\in L^2(\gamma_{-1})$ with compact support. However, neither $\Psi_{\lambda,\sigma}$ nor $\Phi_{\lambda,\sigma}$ are in $L^2(\gamma_{-1})$. If they were, indeed, by Lemma~\ref{Psi} (1) they would be in the kernels of $\As +\lambda I$ and $(\As+\lambda I)^2$, respectively, which contain only the null function.

Therefore, if $f \in L^2(\gamma_{-1})$ has compact support in a ball $B$, we shall denote the integrals $\int\Psi_{\lambda,\sigma} f \,\dd \gamma_{-1}$ and $\int\Phi_{\lambda,\sigma} f \,\dd \gamma_{-1}$ by $(\Psi_{\lambda,\sigma}, f)_{L^2(B,\gamma_{-1})}$ and $(\Phi_{\lambda,\sigma}, f)_{L^2(B,\gamma_{-1})}$, respectively, to emphasize that they \emph{are not} inner products in $L^2(\R^n,\gamma_{-1})$.

\begin{corollary}\label{corlnotm} Let $B$ be an admissible ball.
\begin{itemize}
\item[\emph{(1)}] If $\psi$ is a function in $C_c^\infty(B)$ with integral zero with respect to $\gamma_{-1}$, then $(\As +\lambda I)\psi$ is a multiple of an $X^1_\lambda$-atom.
\item[\emph{(2)}] If $\lambda \neq \mu$ there exists an $X^1_\lambda$-atom which 
is not an $X^1_\mu$-atom.
\end{itemize}
\end{corollary}
\begin{proof}
To prove (1) we only need to show that $(\As +\lambda I)\psi\in q^2_\lambda(B)^\perp$. Indeed, if $v\in q^2_\lambda(B)$ and $(\As+\lambda I)v=c$ on $B$,
\[
(v,(\As +\lambda I)\psi)_{L^2(B,\gamma_{-1})}= ((\As +\lambda I)v,\psi)_{L^2(B,\gamma_{-1})}= (c,\psi)_{L^2(B,\gamma_{-1})}=0.
\]
To prove (2), observe that by Lemma~\ref{Psi}, (3) there exists a function $\psi\in C_c^\infty(B)$ with integral zero with respect to $\gamma_{-1}$ that is not orthogonal to 
$\Psi_{\mu,\sigma}$. By (1) the function $(\As +
\lambda I)\psi$ is, up to a constant factor, an $X^1_\lambda$-atom, but $(\As +\lambda I)\psi \notin q^2_\mu(B)^\perp$ since by 
integration by parts
\[
((\As +\lambda I)\psi, \Psi_{\mu,\sigma})_{L^2(B,\gamma_{-1})}= (\lambda -\mu)(\psi,\Psi_{\mu,\sigma})_{L^2(B,\gamma_{-1})}\neq 0
\]
and $\Psi_{\mu,\sigma} \in h^2_\mu(B)\subset q^2_\mu(B)$ by Lemma~\ref{Psi}, (1).
\end{proof}

The following two lemmata highlight the importance of the functions $\Psi_{\lambda,\sigma}$ and $\Phi_{\lambda,\sigma}$. In particular, Lemma~\ref{lemmaunbounded0} below concerns their role in the unboundedness results of the imaginary powers, while Lemma~\ref{lemmaunbounded} that for the Riesz transforms. They will be used in Theorems~\ref{teo:X1Impow} and~\ref{teo:X1Riesz} respectively. 

\begin{lemma}\label{lemmaunbounded0}
Let $B=B(0,1)$, $u\in \R\setminus \{0\}$ and $f\in L^1(B)$ with $\supp(f)\subseteq \bar{B}$. If there exists $\sigma_0 \in S^{n-1}$ such that $(\Psi_{0,\sigma_0},f)_{L^2(B,\gamma_{-1})}\neq 0$, then $\As^{iu} f \notin L^1(\gamma_{-1})$.
\end{lemma}

\begin{proof}
Let $f$ be as in the statement. Then, for every $x\notin \bar{B}$
\begin{align*}
\As^{iu}& f(x) = 
%\int_{\R^n} k_{\As^{iu}}(x,y)f(y)\gamma_{-1}(y)\, \dd y \\&= 
c(u,n) e^{-|x|^2}\int_{B} \int_0^1 \frac{r^{n-1}(-\log r)^{-iu-1}}{(1-r^2)^{n/2}} e^{-\frac{|rx-y|^2}{1-r^2}} f(y)e^{|y|^2}\, \dd r\, \dd y
\end{align*}
so that
\begin{align*}
\|\As^{iu}f\|_{L^1(\gamma_{-1})}
%&= c(u,n) \int_{\R^n} \abs*{\As^{iu}a(x)}\, e^{|x|^2} dx 
%\\& \geq c(u,n) \int_{(5B)^c} \abs*{\As^{iu}a(x)}\, e^{|x|^2}dx
%\\& 
&\geq c(u,n) \bigg\{\int_{(5B)^c} \bigg|\int_{B} \int_0^{\frac{1}{2}} \dots \, \dd r\, \dd y\bigg| \, \dd x -\int_{(5B)^c} \bigg|\int_{B} \int_{\frac{1}{2}}^1 \dots\, \dd r\, \dd y\bigg|\, \dd x\bigg\}
\\&= c(u,n)\{ I_1 -I_2\}.
\end{align*}
The choice of $5B$ is merely technical. We shall prove that $I_1=\infty$ while $I_2<\infty$. As for $I_2$, since $(x,y)\leq |x|$ for $y\in B$, we get
\begin{align*}
I_2\leq \int_{(5B)^c} \int_{B} \int_{1/2}^1 \frac{(-\log r)^{-1}}{(1-r^2)^{n/2}} e^{-\frac{r^2|x|^2}{1-r^2}} e^{-\frac{r^2|y|^2}{1-r^2}} e^{\frac{2r|x|}{1-r^2}}|f(y)|\, \dd r\, \dd y \, \dd x. 
\end{align*}
Since $e^{-r^2|y|^2/(1-r^2)} \leq 1$ and $f\in L^1(B)$, by changing the order of integration and passing to spherical coordinates
\begin{multline*}
I_2 
%\leq \int_{1/2}^1 \frac{(-\log r)^{-1}}{(1-r^2)^{n/2}} \int_{B(0,5)^c} e^{-\frac{r^2|x|^2}{1-r^2}} \bigg[ e^{\frac{2r|x|}{1-r^2}} - e^{-\frac{2r|x|}{1-r^2}} \bigg]\, \dd x \, \dd r \\&
\leq \int_{1/2}^1 \frac{(-\log r)^{-1}}{(1-r^2)^{n/2}} \int_{(5B)^c} e^{-\frac{r^2|x|^2}{1-r^2}+ \frac{2r|x|}{1-r^2}}\, \dd x \, \dd r  \lesssim \int_{1/2}^1 \frac{(-\log r)^{-1}}{(1-r^2)^{n/2}} \int_5^\infty e^{-\frac{r^2\rho^2}{1-r^2}+ \frac{2r\rho}{1-r^2}}\rho^{n-1}\, \dd \rho \, \dd r.
\end{multline*}
Observe now that for $r\in (1/2,1)$ and every $\rho\geq 0$
\[
-\frac{r^2\rho^2}{1-r^2}+ \frac{2r\rho}{1-r^2} \leq -\frac{r\rho^2}{2(1-r^2)}+ \frac{2r\rho}{1-r^2} = \frac{r(-2\rho^2 +8\rho) }{4(1-r^2)}\leq \frac{4r}{1-r^2} -\frac{r}{4(1-r^2)}\rho^2
\]
since $-2\rho^2 +8\rho \leq 16 -\rho^2$. Observe also that for every $s_0>0$ there is a constant $c$ depending only on $s_0$ and $n$ such that $\exp(-s^2/4)s^{n-1} \leq c \,s\exp(-s^2/5)$ for every $s\geq s_0$. Therefore, since $\rho \sqrt{\frac{r}{1-r^2}}\geq 5\sqrt{2/3}$, there exists a $c$ such that
\begin{align*}
\exp\bigg(-\frac{r}{4(1-r^2)}\rho^2\bigg) \rho^{n-1} \leq c \bigg(\frac{1-r^2}{r}\bigg)^{(n-1)/2} \exp\bigg(-\frac{r}{5(1-r^2)}\rho^2\bigg)\rho \sqrt{\frac{r}{1-r^2}}.
\end{align*}
Thus
\begin{align*}
I_2& \lesssim \int_{1/2}^1 \frac{ e^{\frac{4r}{1-r^2}}}{(-\log r)} \int_5^\infty \frac{\rho r}{1-r^2}e^{-\frac{r}{5(1-r^2)}\rho^2} \, \dd \rho \, \dd r = c \int_{1/2}^1 \frac{e^{\frac{4r}{1-r^2}}}{(-\log r)}e^{-\frac{5r}{1-r^2}}\, \dd r<\infty.
\end{align*}
We now look at $I_1$, which we write as $I_1= \int_5^\infty I_1^1(\rho) \dd\rho$, where
\[
I_1^1(\rho)=\rho^{n-1} \int_{S^{n-1}}\bigg|\int_{B} \int_0^{1/2} \frac{r^{n-1}(-\log r)^{-iu-1}}{(1-r^2)^{n/2}} e^{-\frac{r^2 \rho^2}{1-r^2}} e^{-\frac{r^2 |y|^2}{1-r^2}} e^{\frac{2r\rho(\sigma, y)}{1-r^2}} f(y)\, \dd r\, \dd y \bigg|\, \dd \Sigma(\sigma).
\]
We shall find the asymptotic behaviour of $I_1^1(\rho)$ when $\rho \to \infty$. We make the substitution $r^2 \rho^2/(1-r^2)=t$ in the integral over $r$, and get
$$
I_1^1(\rho) = \frac{1}{\rho \log(\rho^2)}\int_{S^{n-1}} \bigg| \int_0^{\rho^2/3} F(\rho,t,\sigma) \, \dd t\bigg| \, \dd \Sigma(\sigma)
$$
where
\[
F(\rho,t,\sigma)= 2^{iu} \bigg[ \frac{\log(1+\rho^2/t)}{\log \rho^2}\bigg]^{-iu-1}e^{-t} t^{(n-2)/2} \bigg( \int_{B} e^{-\frac{t|y|^2}{\rho^2}}e^{2(\sigma,y) \sqrt{t(1+t/\rho^2)}}\, f(y)\, \dd y\bigg).
\]
We want to apply the dominated convergence theorem for $\rho \to \infty$. Indeed, observe that
\begin{equation}\label{logestimate}
\bigg|\frac{\log \rho^2}{\log (1+\rho^2/t)}\bigg|
 \leq 1+ \frac{|\log t|}{\log 4}
\end{equation}
for every $t\in (0,\rho^2/3)$ and since $t/\rho^2\leq 1/3$,
\begin{equation*}
\abs*{(\sigma,y)\sqrt{t(1+t/\rho^2)}} \leq c\, \sqrt{t}
\end{equation*}
for every $\sigma \in S^{n-1}$ and $y\in B$. Therefore, since $f\in L^1(B)$, for every $t\in (0,\rho^2/3)$
\begin{align*}
|F(\rho,t, \sigma)| &\leq \bigg(1+ \frac{\abs{\log t}}{\log 4}\bigg)e^{-
t} t^{(n-2)/2}e^{c\sqrt{t}}\int_{B} \abs*{ f(y)}\, \dd y 
\\& \lesssim \bigg(1+ \frac{\abs{\log t}}{\log 4}\bigg)e^{-t} t^{(n-2)/2} e^{c\sqrt{t}} 
\eqqcolon g(t).
\end{align*}
Since the function $g$ is integrable on $(0,\infty)\times 
S^{n-1}$, by dominated convergence
\begin{align*}
\lim_{\rho \to \infty} \int_{S^{n-1}} \bigg| \int_0^{\rho^2/3} F(\rho,t,\sigma) \, \dd t\bigg| \, \dd \Sigma(\sigma)
%& = \int_{S^{n-1}}\bigg|\int_0^\infty e^{-t}t^{(n-2)/2} \bigg(\int_{B}e^{2(\sigma,y) \sqrt{t}} f(y)\, \dd y \bigg) \, \dd t\bigg| \, \dd \Sigma(\sigma)\\& 
= \int_{S^{n-1}}
\abs*{(\Psi_{0,\sigma},f)_{L^2(B,\gamma_{-1})}} \, \dd \Sigma(\sigma).
\end{align*}
The integral over $S^{n-1}$ is strictly positive, since the function $
\sigma \mapsto (\Psi_{0,\sigma},f)_{L^2(B,\gamma_{-1})}$ is continuous and by assumption there exists a $\sigma_0\in S^{n-1}$ such that $
(\Psi_{0,\sigma_0},f)_{L^2(B,\gamma_{-1})}\neq 0$. Therefore
$$
I_1^1(\rho) \sim \frac{1}{\rho \log (\rho^2)}\int_{S^{n-1}}
\abs*{(\Psi_{0,\sigma},f)_{L^2(B,\gamma_{-1})}} \, \dd \Sigma(\sigma) \quad \mbox{for $\rho \to \infty$},
$$
and
$$
I_1=\int_5^\infty I^1_1(\rho)\,\dd \rho=\infty.
$$
This concludes the proof.
\end{proof}

The following lemma is the counterpart of Lemma~\ref{lemmaunbounded0} for the Riesz transforms, but it requires a more sophisticated analysis.
\begin{lemma}\label{lemmaunbounded}
Let $B=B(0,1)$, $j=1,\dots, n$ and $f\in L^1(B)$ with $\supp(f)\subseteq \bar{B}$.
\begin{itemize}
\item[\emph{(i)}] For every $\lambda \in [0,1]$, if there exists $\sigma_0 \in S^{n-1}$ such that $(\Psi_{\lambda,\sigma_0},f)_{L^2(B,\gamma_{-1})}\neq 0$, then $(\Rs_\lambda)_j f \notin L^1(\gamma_{-1})$.
\item[\emph{(ii)}] For every $\lambda \in [0,1)$, if $(\Psi_{\lambda,\sigma},f)_{L^2(B,\gamma_{-1})}=0$ for every $\sigma \in S^{n-1}$ but there exists $\sigma_0 \in S^{n-1}$ such that $(\Phi_{\lambda,\sigma_0},f)_{L^2(B,\gamma_{-1})}\neq 0$, then $(\Rs_\lambda)_j f\notin L^1(\gamma_{-1})$.
\end{itemize}
\end{lemma}
Notice that, while statement (i) holds for all $\lambda\in[0,1]$, statement (ii) does not hold for $\lambda=1$.
\begin{proof}
Let $\lambda\geq0$. For almost every $x\in \R^n$
\begin{align*}
(\Rs_\lambda)_j f(x) =c(n) e^{-|x|^2} \int_{B}\int_0^1 \frac{r^{n+\lambda-1}(-\log r)^{-1/2}}{(1-r^2)^{(n+2)/2}}(x_j-ry_j) e^{-\frac{|rx-y|^2}{1-r^2}}e^{|y|^2}f(y)\, \dd r\, \dd y,
\end{align*}
so that
\begin{align*}
\|(\Rs_\lambda)_j f\|_{L^1(\gamma_{-1})}&
%= c(n) \int_{(5B)^c} \abs*{(\Rs_\lambda)_j f(x)}\, e^{|x|^2}dx\\& 
\gtrsim \int_{(5B)^c} \bigg|\int_{B} \int_0^{\frac{1}{2}} \dots \, \dd r\, \dd y\bigg| \, \dd x -\int_{(5B)^c} \bigg|\int_{B} \int_{\frac{1}{2}}^1 \dots\, \dd r\, \dd y\bigg|\, \dd x\\&= J_1^\lambda -J_2^\lambda.
\end{align*}
The finiteness of $J_2^\lambda$ for every $\lambda\geq 0$ can be seen exactly as that of $I_2$ in the proof of Lemma~\ref{lemmaunbounded0}, and we omit the details. As for $J_1^\lambda$, 
\begin{align}\label{J1}
J_1^\lambda&= \int_5^\infty \rho^{n-1} \int_{S^{n-1}}\bigg|\int_{B} \int_0^{1/2} \frac{r^{n+\lambda-1}(-\log r)^{-1/2}}{(1-r^2)^{(n+2)/2}} e^{-\frac{r^2 \rho^2}{1-r^2}} e^{-\frac{r^2 |y|^2}{1-r^2}} \nonumber \\& \qquad \qquad\qquad\times(\rho \sigma_j-ry_j)e^{\frac{2r\rho(\sigma, y)}{1-r^2}}\, \dd rf(y)\, \dd y \bigg|\, \dd\Sigma(\sigma) \, \dd \rho = \int_5^\infty J_{1,1}^\lambda(\rho) \, \dd\rho.
\end{align}
We shall describe the asymptotic behaviour of $J_{1,1}^\lambda(\rho)$ when $\rho \to \infty$. We perform the change of variables $r^2\rho^2/(1-r^2)=t$, and get
\begin{align}\label{J11lambda}
J_{1,1}^\lambda(\rho)
%&= \rho^{n-1} \int_{S^{n-1}} \bigg|\frac{c}{\rho^{n+\lambda-1} \log(\rho^2)^{1/2}}\int_0^{\rho^2/3} \bigg[ \frac{\log(1+\rho^2/t)}{\log \rho^2}\bigg]^{-1/2}\frac{e^{-t} t^{(n+\lambda-2)/2}}{(1+t/\rho^2)^{1+\lambda/2}} \\& \qquad \qquad \times \bigg( \int_{B} e^{-\frac{t|y|^2}{\rho^2}} \bigg[\bigg(\sigma_i - \frac{y_j}{\rho^2}\sqrt{\frac{t}{1+t/\rho^2}}\bigg) e^{2(\sigma,y) \sqrt{t(1+t/\rho^2)}} \bigg]f(y)\, \dd y\bigg) \, \dd t\bigg|\, \dd\Sigma(\sigma) \\&
= \frac{c}{\rho^{\lambda} \sqrt{\log \rho}}\int_{S^{n-1}} \bigg| \int_0^{\rho^2/3} h_\lambda(\rho,t,\sigma) \, \dd t\bigg| \, \dd\Sigma(\sigma)
\end{align}
where
\begin{multline*}
h_\lambda(\rho, t,\sigma) = \bigg[ \frac{\log(1+\rho^2/t)}{\log \rho^2}\bigg]^{-1/2}\frac{e^{-t} t^{(n+\lambda-2)/2}}{(1+t/\rho^2)^{\lambda/2}} \\ \times \int_{B} e^{-\frac{t|y|^2}{\rho^2}} \bigg(\sigma_j - \frac{y_j}{\rho^2}\sqrt{\frac{t}{1+t/\rho^2}}\bigg) e^{2(\sigma,y) \sqrt{t(1+t/\rho^2)}} f(y)\, \dd y.
\end{multline*}
Thus, the asymptotic behaviour of $J_{1,1}^\lambda(\rho)$ when $\rho \to \infty$ can be recovered by that of the inner integral in~\eqref{J11lambda}. We shall need the first \emph{two} terms of the asymptotic expansion of $J_{1,1}^\lambda$, since the first term is not enough to prove the statement (ii) of the theorem.

We claim that $J_{1,1}^\lambda(\rho)$ equals
\begin{equation}\label{claimLemma}
\frac{c}{\rho^{\lambda} \sqrt{\log \rho}}\int_{S^{n-1}} \bigg|\sigma_j\bigg((\Psi_{\lambda, \sigma}, f)_{L^2(B,\gamma_{-1})}+ \frac{1}{4\log(\rho)} (\Phi_{\lambda, \sigma}, f)_{L^2(B,\gamma_{-1})}\bigg) + R(\sigma,\rho)\bigg| \, \dd \Sigma(\sigma)
\end{equation}
where $c\neq 0$ and $|R(\sigma,\rho)|\leq C/\log^2 \rho$.

Assuming the claim for the moment, we complete the proof. If there exists $\sigma_0 \in S^{n-1}$ such that $(\Psi_{\lambda,\sigma_0},f)_{L^2(B,\gamma_{-1})}\neq 0$, then, by continuity, we can find an open subset $U$ of $S^{n-1}$ such that $|\sigma_j|\geq \epsilon >0$ and $(\Psi_{\lambda,\sigma_0},f)_{L^2(B,\gamma_{-1})} \neq 0$ for every $\sigma \in U$. Thus the integral in~\eqref{claimLemma} is bounded below by 
\[
\int_{U}\bigg|\sigma_j\bigg((\Psi_{\lambda, \sigma}, f)_{L^2(B,\gamma_{-1})}+ \frac{1}{4\log(\rho)} (\Phi_{\lambda, \sigma}, f)_{L^2(B,\gamma_{-1})}\bigg) + R(\sigma,\rho)\bigg| \, \dd \Sigma(\sigma)\geq C>0.
\]
Hence $J_{1,1}^\lambda(\rho) \geq \frac{C}{\rho^\lambda \sqrt{\log \rho}}$ for all $\rho$ sufficiently large, and the integral $\int_5^\infty J_{1,1}^\lambda(\rho)\, \dd \rho$ diverges for all $\lambda \in [0,1]$. If $(\Psi_{\lambda,\sigma},f)_{L^2(B,\gamma_{-1})}=0$ for every $\sigma \in S^{n-1}$ but there exists $\sigma_0 \in S^{n-1}$ such that $(\Phi_{\lambda,\sigma_0},f)_{L^2(B,\gamma_{-1})}\neq 0$, the same continuity arguments used previously shows that there exists an open subset $U$ of $S^{n-1}$ such that $|\sigma_j|\geq \epsilon >0$ and $(\Phi_{\lambda,\sigma_0},f)_{L^2(B,\gamma_{-1})}\neq 0$ for every $\sigma \in U$. Thus, the integral in~\eqref{claimLemma} is bounded below by
\[
\int_{U}\bigg|\sigma_j \frac{1}{4\log(\rho)} (\Phi_{\lambda, \sigma}, f)_{L^2(B,\gamma_{-1})} + R(\sigma,\rho)\bigg| \, \dd \Sigma(\sigma)\geq \frac{c}{\log \rho}>0
\]
for all $\rho$ sufficiently large. Thus $J_{1,1}^\lambda(\rho) \geq \frac{c}{\rho^\lambda \log^{3/2} \rho}$ and $\int_5^\infty J_{1,1}^\lambda(\rho)\, \dd \rho$ diverges for all $\lambda \in [0,1)$.

It remains to prove the claim. To do this, we use Taylor formula. First of all, observe that for every $t \in (0,\rho^2/3)$ we have $t/\rho^2 \leq 1/3<1$. Then
\[e^{-t|y|^2/\rho^2} = 1+ R_1(y,t,\rho),\qquad (1+t/\rho^2)^{-\lambda/2} = 1+R_2^\lambda(t,\rho),
 \]
where $|R_1(y,t,\rho)|\leq C\, t/\rho^2$ and $|R_2^\lambda(t,\rho)|\leq C(\lambda)\, t/\rho^2$ for some constants $C$ and $C(\lambda)$ which do not depend on $t$, $\rho$, $y$ or $\sigma$. The logarithmic term is more delicate, and it will give the main contributions to the integral. The first step is to write
\[\frac{\log(1+\rho^2/t)}{\log (\rho^2) }=1 + \frac{\log (1/t +1/\rho^2)}{\log(\rho^2)} =1 + \frac{\log (1/t)}{2 \log \rho } + \frac{\log (1+t/\rho^2)}{2\log \rho},
\]
and to observe that for every $s>-1$
\[
(1+s)^{-1/2} = 1-\frac{1}{2} s + R(s), \qquad |R(s)| \leq s^2 \max \bigg(1, \frac{1}{1+s}\bigg).
\]
In particular, if $s= \log (1/t+1/\rho^2)/\log(\rho^2)$, then $s>-1$ and
\[
\frac{1}{1+s} = \frac{\log(\rho^2)}{\log(1+\rho^2/t)} \leq 1+ \frac{|\log t|}{\log 4}
\]
by~\eqref{logestimate}. Observe moreover that $\log(1+t/\rho^2)\leq t/\rho^2$. Therefore,
\[
\bigg[\frac{\log(1+\rho^2/t)}{\log (\rho^2)}\bigg]^{-1/2} = 1 - \frac{1}{4}\frac{\log (1/t)}{\log \rho } + R(t,\rho), \quad |R(t,\rho)|\leq C\,(1+|\log^3 t|)/\log^2 \rho,\]
where $C$ is an absolute constant. Finally, observe that for every $y\in B$, $\sigma\in S^{n-1}$ and $\rho\geq 5$
\[\sigma_j - \frac{y_j}{\rho^2}\sqrt{\frac{t}{1+t/\rho^2}} = \sigma_j + R_3(y,t,\rho), \quad |R_3(y,t,\rho)|\leq \frac{\sqrt{t}}{\rho^2}.\]
By substituting the expansions above, then, one can see that
\begin{align}\label{splitA1A2}
\int_0^{\rho^2/3 }& h_\lambda(\rho, t, \sigma)\, \dd t = A_1^\lambda(\sigma,\rho) - A_2^\lambda(\sigma,\rho) + R_0(\sigma,\rho),
\end{align}
where
\[
A_1^\lambda(\sigma,\rho) = \sigma_j\int_{0}^{\rho^2/3}e^{-t}t^{(n+\lambda-2)/2} \bigg(\int_{B} e^{2(\sigma,y) \sqrt{t(1+t/\rho^2)}} f(y)\, \dd y \bigg) \, \dd t,
\]
\[
A_2^\lambda(\sigma,\rho) = \frac{\sigma_j}{4\log \rho} \int_{0}^{\rho^2/3} e^{-t}t^{(n+\lambda-2)/2} \log(1/t)\bigg(\int_{B} e^{2(\sigma,y) \sqrt{t(1+t/\rho^2)}} f(y)\, \dd y \bigg) \, \dd t 
\]
and $|R_0(\sigma,\rho)|\leq C/\log^2 \rho$. We first concentrate on $A_1^\lambda$, that we split as
\[
A_1^\lambda(\sigma,\rho) = \int_0^\rho \dots \, \dd t + \int_\rho^{\rho^2/3} \dots \, \dd t.
\]
Then
\[
\abs*{\int_\rho^{\rho^2/3}e^{-t}t^{(n+\lambda-2)/2} \bigg(\int_{B} e^{2(\sigma,y) \sqrt{t(1+t/\rho^2)}} f(y)\, \dd y \bigg) \, \dd t}\leq C e^{-\rho/2}
\]
while
\begin{multline*}
\int_0^\rho e^{-t}t^{(n+\lambda-2)/2} \bigg(\int_{B} e^{2(\sigma,y) \sqrt{t(1+t/\rho^2)}} f(y)\, \dd y \bigg) \, \dd t \\= \int_0^\rho e^{-t}t^{(n+\lambda-2)/2} \bigg(\int_{B} e^{2(\sigma,y) \sqrt{t}} f(y)\, \dd y \bigg) \, \dd t + R'(\sigma,\rho)
\end{multline*}
with $|R'(\sigma,\rho)|\leq C/\sqrt{\rho}$, since
\begin{align*}
\abs*{e^{2(\sigma,y) \sqrt{t(1+t/\rho^2)}} -e^{2(\sigma,y) \sqrt{t}}}\leq C \, e^{2\sqrt{t}}/\sqrt{\rho}
\end{align*}
for $t\in (0,\rho)$, $\sigma\in S^{n-1}$ and $y\in B$. Finally, write
\[
\int_0^\rho e^{-t}t^{(n+\lambda-2)/2} \bigg(\int_{B} e^{2(\sigma,y) \sqrt{t}} f(y)\, \dd y \bigg) \, \dd t = \int_0^\infty \dots \, \dd t - \int_\rho^\infty \dots \, \dd t
\]
and observe that the first integral in the right-hand side is exactly $(\Psi_{\lambda,\sigma},f)_{L^2(B,\gamma_{-1})}$. Instead, for $\rho$ sufficiently large
\[\abs*{\int_{\rho}^\infty e^{-t}t^{(n+\lambda-2)/2}\bigg(\int_{B} e^{2(\sigma,y) \sqrt{t}} f(y)\, \dd y \bigg) \, \dd t} \leq C \, e^{-\rho/2}\]
for some absolute constant $C>0$. Therefore
\[
A_1^\lambda(\sigma,\rho)=\sigma_j (\Psi_{\lambda, \sigma}, f)_{L^2(B,\gamma_{-1})} + R_1(\sigma,\rho)
\]
with $|R_1(\sigma,\rho)|\leq C/\sqrt{\rho}$. Similar arguments apply to $A_2^\lambda$, and yield
\[
A_2^\lambda(\sigma,\rho)= \sigma_j \frac{1}{4\log(\rho)} (\Phi_{\lambda, \sigma}, f)_{L^2(B,\gamma_{-1})} + R_2(\sigma,\rho)
\]
with $|R_2(\sigma,\rho)| \leq C/\sqrt{\rho}$. Therefore, by~\eqref{splitA1A2}
\[
\int_0^{\rho^2/3 }h_\lambda(\rho, t, \sigma)\, \dd t =\sigma_j\bigg[(\Psi_{\lambda, \sigma}, f)_{L^2(B,\gamma_{-1})}+ \frac{1}{4\log(\rho)} (\Phi_{\lambda, \sigma}, f)_{L^2(B,\gamma_{-1})}\bigg] + R(\sigma,\rho)
\]
with $|R(\sigma,\rho)| \leq C/\log^2(\rho)$, and the claim is proved.
\end{proof}

\section{Weak type $(1,1)$}\label{Sec:weakOUrovesciato}
In this section, we prove the following theorem.
\begin{theorem}\label{teo:weaktype}
Let $\lambda \geq0$ and $u\in \R\setminus \{0\}$. Then
\begin{itemize}
\item[\emph{(i)}] the imaginary powers $(\As+\lambda I)^{iu}$ are of weak type $(1,1)$ for every $\lambda\geq 0$;
\item[\emph{(ii)}] the Riesz transforms $\Rs_\lambda$ are of weak type $(1,1)$ for every $\lambda \geq 1$.
\end{itemize}
\end{theorem} 
The proof of Theorem~\ref{teo:weaktype} is inspired by the proof of~\cite[Theorem 1.1]{B}. However, there is some significant difference and we shall include all the necessary details.
\par

Since $\gamma_{-1}$ is locally doubling on admissible balls, but not globally doubling, we start by splitting $\R^n \times \R^n$ into a local and a global region. Recall that for $\delta>0$
\begin{equation}\label{LandG}
N_\delta = \bigg\{(x,y)\in \R^n\times \R^n \colon |x-y|\leq \frac{\delta}{1+|x|+|y|}\bigg\}
\end{equation}
by~\eqref{Ns_prima} and define $G\coloneqq N_1^c$. The regions $N_1$ and $N_2$ will be called \emph{local} regions, while $G$ will be the \emph{global} region. Moreover, fix a smooth function $\chi\colon \R^{n}\times \R^n \to \R$ such that 
\[\chi_{N_1}\leq \chi \leq \chi_{N_2},\qquad |\nabla_x \chi(x,y)|+|\nabla_y \chi(x,y)|\leq \frac{C}{|x-y|}\quad \mbox{for every } x\neq y.\]
For any operator $T$, bounded on $L^2(\gamma_{-1})$, with Schwartz kernel $K_T$ we define
\[
K_{T,\loc}\coloneqq \chi K_{T}, \qquad K_{T,\glob}\coloneqq K_{T}-K_{T,\loc}.\]
We shall denote the operators with kernel $K_{T,\loc}$ and $K_{T,\glob}$ by $T_{\loc}$ and $T_{\glob}$ respectively. Of course $T= T_{\loc}+ T_{\glob}$. Therefore, to prove the weak type $(1,1)$ of $T$, it will be enough to prove the weak type $(1,1)$ of both $T_\loc$ and $T_\glob$. \par
The proof for $\Rs_{\lambda, \loc}$ and $(\As +\lambda I)^{iu}_\loc$ will be rather standard, since by~\cite[Theorem 3.2.8]{Salogni} we can reduce to proving some Calder\'on-Zygmund type estimates for their kernel.

As for the global parts, we prove that there exists a kernel $\bar{K}$, related to the Mehler maximal kernel, which controls both the kernels of $\Rs_{\lambda,\glob}$, $\lambda\geq 1$, and of $(\As+\lambda I)^{iu}_{\glob}$, $\lambda\geq 0$, and which is the kernel of an operator of weak type $(1,1)$.

To shorten the notation, for $x,y\in \R^n$ we set
\[
\alpha\coloneqq |x-y||x+y|,\qquad \beta\coloneqq \frac{|x-y|}{|x+y|}, \qquad \eta(x,y)\coloneqq e^{\frac{|x|^2}{2}-\frac{|y|^2}{2}- \frac{|x-y||x+y|}{2}}.
\]
We also denote by $\theta=\theta(x,y)$ the angle between $x$ and $y$, and by $\theta'$ the angle between $y-x$ and $y+x$. Observe that $\beta<1$ if and only if $(x,y)>0$. 

We begin by stating a lemma which is essentially~\cite[Lemma 3.1]{B}. Its proof is elementary and omitted.
\begin{lemma}\label{lemmaprel}
Let $(x,y)\in \R^n$. Then
\begin{itemize}
\item[\emph{(1)}] if $(x,y)\in G$ and $\beta<1$, then $\alpha\geq 1/4$.
\item[\emph{(2)}] if $(x,y)\in G$, then $|x-y|\geq \frac{1}{2}(1+|x|)^{-1}$.
\item[\emph{(3)}] $|x \pm y|\geq |x|\sin\theta$. In particular, $\alpha\geq |x|^2\sin^2\theta$.
\item[\emph{(4)}] ${|x|^2} -{|y|^2} -{|x-y||x+y|}\leq 0$.
\item[\emph{(5)}] $\frac{|x|^2}{2} - \frac{|y|^2}{2} - \frac{|x+y||x-y|}{2}=\frac{-2|x|^2|y|^2 \sin^2\theta}{|x-y||x+y|(1-\cos \theta')}$.
\end{itemize}
\end{lemma}
The following lemma allows us to obtain Calder\'on-Zygmund type estimates in the local region for the kernels of the imaginary powers and the Riesz transforms of $\As$. Its proof is an \emph{almost verbatim} repetition of~\cite[Lemma 3.6]{B}, and is omitted.
\begin{lemma}\label{lemmalocal}
Let $\mu,\nu\geq 0$ be such that $\mu >\nu +1$. Then, for every $(x,y)\in N_2$, $x\neq y$
\[
K_{\mu,\nu}(x,y)\coloneqq \int_0^1 \frac{|x-ry|^\nu }{(1-r^2)^{\frac{n+\mu}{2}}}e^{-\frac{|x-ry|^2}{1-r^2}}\, \dd r \leq \frac{C}{|x-y|^{n+\mu-\nu-2}}.
\]
\end{lemma}
\subsection{The Mehler Maximal Operator}\label{Sec:max}
To prove that the global parts of the imaginary powers and of the Riesz transforms are of weak type $(1,1)$, we shall prove that they are controlled by the integral operator $\mathcal{\bar{K}}$ whose kernel with respect to the Lebesgue measure is
\begin{equation*}
\bar K (x,y) \coloneqq e^{-|x|^2+|y|^2} \bigg( \frac{|x+y|}{|x-y|}\bigg)^{n/2} e^{\frac{|x|^2}{2}-\frac{|y|^2}{2}- \frac{|x-y||x+y|}{2}} \Phi(x,y) \chi_G(x,y)
\end{equation*}
where 
\begin{equation*}
\Phi(x,y)=
\begin{cases}
\frac{1}{\alpha^{n/2}} \qquad &\mbox{if } \beta\geq 1\\
\frac{1}{\alpha^{n/2}} + (1-\beta)^n &\mbox{if }\beta< 1.
\end{cases}
\end{equation*}
The kernel $\bar{K}$ appears naturally when one tries to estimate the {\it maximal Mehler kernel} 
\[
H^*(x,y)\coloneqq \sup_{t>0} H_t(x,y)
\]
in the global region, as we shall see in more detail in Proposition~\ref{maxglobal}.
\begin{lemma}\label{barK}
The operator $\bar{\mathcal{K}}$ is of weak type $(1,1)$.
\end{lemma}
\begin{proof}
 The operator with kernel 
$$
e^{-|x|^2+|y|^2}
(1+|x|)^n \wedge (|x|\sin \theta)^{-n}
$$
is of weak type $(1,1)$ by \cite[Lemma 3.3.4]{Salogni} (see also 
\cite[Lemma 4.4]{GMMST}).
Thus, it is enough to prove
\[
\bigg( \frac{|x+y|}{|x-y|}\bigg)^{n/2} \eta(x,y) \Phi(x,y) \leq C (1+|x|)^n \wedge (|x|\sin \theta)^{-n}
\]
for every $(x,y)\in G$. First consider the inequality involving $(1+|x|)^n$. We consider the cases $\beta<1$ and $\beta \geq 1$ separately.

\textbf{1.} If $\beta\geq1$, observe that
\[
\bigg( \frac{|x+y|}{|x-y|}\bigg)^{n/2} \eta(x,y) \Phi(x,y) = \frac{\eta(x,y)}{|x-y|^n} \leq \frac{1}{|x-y|^n}\lesssim (1+|x|)^n,
\]
the last inequality by Lemma~\ref{lemmaprel}, (2).

\textbf{2.} If $\beta<1$, by Lemma~\ref{lemmaprel}, (1) we have $\Phi(x,y)\leq C$ for every $(x,y)\in G$. Thus, we only have to prove that
\[
\bigg( \frac{|x+y|}{|x-y|}\bigg)^{n/2} \leq (1+|x|)^n.
\]
If $|y|\leq 2|x|$, by Lemma~\ref{lemmaprel}, (2)
\[
\frac{|x+y|}{|x-y|}\leq \frac{|x|+|y|}{|x-y|}\leq C|x|(1+|x|)\leq (1+|x|)^2.
\]
If $|y|> 2|x|$, we have both $|x-y|\geq |y|-|x| \geq |y|/2$ and $|x-y|\geq |y|-|x|\geq |x|$, so that
\[
\frac{|x+y|}{|x-y|}\leq \frac{|x|}{|x-y|} + \frac{|y|}{|x-y|}\leq C.
\]

We now examine the inequality involving $(|x|\sin\theta)^{-n}$. We again consider the cases $\beta<1$ and $\beta \geq 1$ separately.

\textbf{1'.} If $\beta\geq1$, just observe that by the definition of $\beta$ and Lemma~\ref{lemmaprel}, (3) and (4)
\[
\bigg( \frac{|x+y|}{|x-y|}\bigg)^{n/2} \eta(x,y) \leq 1, \qquad 
\Phi(x,y) = \frac{1}{\alpha^{n/2}} \leq \frac{C}{(|x|\sin\theta)^n}.
\]

\textbf{2'.} If $\beta<1$, we prove separately
\begin{equation}\label{caso1}
\bigg( \frac{|x+y|}{|x-y|}\bigg)^{n/2} \eta(x,y) (1-\beta)^n\leq C(|x|\sin\theta)^{-n}.
\end{equation}
and
\begin{equation}\label{caso2}
\bigg( \frac{|x+y|}{|x-y|}\bigg)^{n/2} \eta(x,y) \frac{1}{\alpha^{n/2}}\leq C(|x|\sin\theta)^{-n}.
\end{equation}
The inequality~\eqref{caso2} is easily seen since
\[\bigg( \frac{|x+y|}{|x-y|}\bigg)^{n/2} \eta(x,y) \frac{1}{\alpha^{n/2}} = \frac{1}{|x-y|^{n}}\eta(x,y) \leq\frac{1} {(|x|\sin\theta)^n},
\]
by Lemma~\ref{lemmaprel}, (3) and (4). As for~\eqref{caso1}, since the function $0\leq u\mapsto u^{n/2}e^{-u}$ is bounded, by Lemma~\ref{lemmaprel}, (5) we obtain
\begin{multline*}
\bigg(\frac{|x+y|}{|x-y|}\bigg)^{n/2} \eta(x,y) 
%= \bigg(\frac{|x+y|}{|x-y|}\bigg)^{n/2} \bigg(\frac{|x-y||x+y|(1-\cos \theta')}{2|x|^2|y|^2 \sin^2\theta}\bigg)^{n/2} \bigg(\frac{2|x|^2|y|^2 \sin^2\theta}{|x-y||x+y|(1-\cos \theta')}\bigg)^{n/2} e^{\frac{-2|x|^2|y|^2 \sin^2\theta}{|x-y||x+y|(1-\cos \theta')}}&
= \bigg(\frac{|x+y|}{|x-y|}\bigg)^{n/2} e^{\frac{-2|x|^2|y|^2 \sin^2\theta}{|x-y||x+y|(1-\cos \theta')}}
\\ \leq \bigg(\frac{|x+y|^2(1-\cos \theta')}{2|x|^2|y|^2\sin^2\theta}\bigg)^{n/2}= C ({|x|\sin\theta})^{-n} \bigg(\frac{|x+y|^2(1-\cos\theta')}{|y|^2}\bigg)^{n/2}.
\end{multline*}
Thus it remains only to prove
\[
(1-\beta)^2\bigg(\frac{|x+y|^2(1-\cos\theta')}{|y|^2}\bigg)\leq C.
\]
If $|x|\leq 2|y|$ this is immediate. Otherwise, an elementary computation shows that
\begin{align*}
(1-\beta)^2\bigg(\frac{|x+y|^2(1-\cos\theta')}{|y|^2}\bigg)= 2 h_\theta(|x|^2/|y|^2),
\end{align*}
where
\[h_\theta(t)= 1+t -\sqrt{(1-t)^2 -4t\sin^2\theta}.\]
Since the functions $h_\theta$ are bounded on $(4,\infty)$ uniformly in $\theta$, the proof is complete.
\end{proof}
As announced previously, we now show that the kernel $\bar{K}$ arises naturally when one tries to estimate 
the maximal Mehler kernel $H^*$ in the global region. The estimate, combined with Lemma~\ref{barK}, provides an alternative proof of the weak type $(1,1)$ of the global part of the maximal operator (see \cite{Salogni})
\begin{equation*}
\mathcal{H}^*f(x)=\sup_{t>0}|e^{-t\As}f(x)|.
\end{equation*}
\begin{proposition}\label{maxglobal}
There exists a constant $C$ such that $H^*(x,y) \leq C \bar K(x,y)$ for every $(x,y)\in G$.
\end{proposition}
\begin{proof}
We begin by applying the rescaling
\begin{equation}\label{taus}
\tau(s)\coloneqq \log \frac{1+s}{1-s},
\end{equation}
which was introduced in~\cite{GMMST}, to $H_t$. Then $H^*(x,y)= \sup_{0<s<1} H_{\tau(s)}(x,y)$. Since
\begin{align*}
H_{\tau(s)}(x,y) = e^{-|x|^2+|y|^2} e^{\frac{|x|^2}{2}-\frac{|y|^2}{2}}\frac{(1-s)^n}{(4s)^{n/2}} e^{-\frac{1}{4}(s|x+y|^2+\frac{1}{s}|x-y|^2)},
\end{align*}
we obtain\[
H^*(x,y) = e^{-|x|^2+|y|^2} \eta(x,y) \sup_{0<s<1}\frac{(1-s)^n}{(4s)^{n/2}} e^{-\frac{1}{4}(s|x+y|^2+\frac{1}{s}|x-y|^2 -2|x-y||x+y|)}.
\] 
After the substitution $s/\beta=\sigma$ in the supremum, we get
\begin{align*}
H^*(x,y)\approx e^{-|x|^2+|y|^2} \bigg( \frac{|x+y|}{|x-y|}\bigg)^{n/2} \eta(x,y)\sup_{0<\sigma< 1/\beta}\frac{(1-\sigma \beta)^n}{\sigma^{n/2}} e^{-\frac{1}{4}\alpha\varphi(\sigma)},
\end{align*}
where
\[\varphi(\sigma)\coloneqq \sigma + \frac{1}{\sigma}-2 = \frac{(\sigma-1)^2}{\sigma}.\]
It remains then to estimate the supremum. Since its argument is a decreasing function of $\beta$, if $\beta\geq1$ we get 
\[
\sup_{0<\sigma< 1/\beta}\frac{(1-\sigma \beta)^n}{\sigma^{n/2}} e^{-\frac{1}{4}\alpha\varphi(\sigma)}\leq \sup_{0<\sigma<\infty}\frac{(1-\sigma)^n}{\sigma^{n/2}} e^{-\frac{1}{4}\alpha\varphi(\sigma)}\lesssim \frac{1}{\alpha^{n/2}}.
\]
If $\beta<1$, observe that
\[
(1-\sigma\beta)^n =[(1-\sigma)+\sigma(1-\beta)]^n \lesssim |1-\sigma|^n + \sigma^n (1-\beta)^n
\]
so that
\begin{align*}
\sup_{0<\sigma< 1/\beta}\frac{(1-\sigma \beta)^n}{\sigma^{n/2}} e^{-\frac{1}{4}\alpha\varphi(\sigma)} &\lesssim \sup_{0<\sigma< \infty} \frac{|1-\sigma|^n}{\sigma^{n/2}} e^{-\alpha\varphi(\sigma)} + (1-\beta)^n\sup_{0<\sigma< \infty} \sigma^{n/2}e^{-\alpha\varphi(\sigma)} \\&\lesssim \frac{1}{\alpha^{n/2}} + (1-\beta)^n.
\end{align*}
Summarizing, we have proved that
\begin{equation}\label{supPhi}
\sup_{0<\sigma< 1/\beta}\frac{(1-\sigma \beta)^n}{\sigma^{n/2}} e^{-\frac{1}{4}\alpha\varphi(\sigma)}\lesssim \Phi(x,y),
\end{equation}
and this completes the proof.
\end{proof}
We are now ready to prove Theorem~\ref{teo:weaktype}.
\subsection{Proof of Theorem~\ref{teo:weaktype}, (ii)} As already said, we treat the local and the global parts of $\Rs_{\lambda}$ separately.
\begin{proposition}\label{proplocalriesz}
For every $\lambda\geq0$, $\Rs_{\lambda,\loc}$ is of weak type $(1,1)$.
\end{proposition}
\begin{proof}
Observe that by Lemma~\ref{lemmalocal}
\[\abs{K_{\Rs_{\lambda,\loc}}(x,y)}\lesssim K_{3,1}(x,y)\lesssim |x-y|^{-n}\]
and that for every $j=1,\dots, n$
\begin{align*}
\abs{\nabla_{x} K_{(\Rs_{\lambda,\loc})_j}(x,y)} + \abs{\nabla_{y} K_{(\Rs_{\lambda,\loc})_j}(x,y)}&\lesssim K_{3,0}(x,y) + K_{5,2}(x,y) + K_{3,1}(x,y)|x-y|^{-1}\\& \lesssim |x-y|^{-(n+1)}
\end{align*}
for every $(x,y)\in N_2$ such that $x\neq y$, $\lambda\geq0$, $j=1,\dots,n$. Therefore, $\Rs_{\lambda, \loc}$ is a Calder\'on-Zygmund operator and the conclusion follows by~\cite[Theorem 3.2.8]{Salogni}
\end{proof}
We now consider the global part. Observe first that for every $\lambda\geq 0$, by~\eqref{tRieszkernel}
\[
|K_{\Rs_\lambda}(x,y)|\lesssim \int_0^\infty \frac{e^{-(n+\lambda)t}}{(1-e^{-2t})^{(n+2)/2}}\frac{|x-e^{-t}y|}{\sqrt{1-e^{-2t}}}e^{-\frac{|x-e^{-t}y|^2}{1-e^{-2t}}} \, \dd t \eqqcolon K_\lambda(x,y),
\]
since $t \geq (1-e^{-2t})/2$ for every $t\geq 0$.
\begin{proposition}\label{rieszglobal}
Let $\lambda \geq 1$. For every $(x,y)\in G$
\begin{equation*}
K_\lambda(x,y)\leq C\bar{K}(x,y).
\end{equation*}
In particular, $\Rs_{\lambda, \glob}$ is of weak type $(1,1)$ for every $\lambda \geq 1$.
\end{proposition}
\begin{proof}
First of all, observe that it is enough to prove the statement for $\lambda=1$, since $K_\lambda \leq K_1$ for $\lambda\geq 1$. The change of variable $t=\tau(s)$ (recall~\eqref{taus}) in the integral defining $K_1$ yields
\begin{align*}
K_{1}(x,y)
%&= \int_0^1 \frac{(1-s)^n}{(4s)^{n/2+1}}\frac{\tau(s)^{-1/2}}{1+s} \abs{(1+s)x -(1-s)y} e^{-\frac{\abs{(1+s)x-(1-s)y}^2}{4s}}\, \dd s \\& 
&\lesssim \int_0^1 \frac{(1-s)^n}{s^{(n+3)/2}} \abs{(1+s)x -(1-s)y} e^{-\frac{\abs{(1+s)x-(1-s)y}^2}{4s}}\, \dd s\\& = e^{-|x|^2+|y|^2} \eta(x,y) \int_0^1 \frac{(1-s)^n}{s^{(n+3)/2}}\abs{(1+s)x -(1-s)y} e^{-\frac{1}{4}\alpha \varphi(s/\beta)}\, \dd s,
\end{align*}
After the change of variable $s/\beta =\sigma$ in the integral we obtain
\begin{multline*}
\int_0^1 \frac{(1-s)^n}{s^{(n+3)/2}}\abs{(1+s)x -(1-s)y} e^{-\frac{1}{4}\alpha \varphi(s/\beta)}\, \dd s \\ = \frac{1}{\beta^{n/2}} \int_0^{1/\beta} \frac{(1-\sigma\beta)^n}{\sigma^{(n+3)/2}}\frac{\abs{(1+\sigma\beta)x - (1-\sigma\beta)y}}{\sqrt{\beta}}e^{-\frac{1}{4}\alpha \varphi(\sigma)}\, \dd \sigma.
\end{multline*}
Since
\begin{align*}
\frac{|(1+\sigma \beta)x -(1-\sigma\beta)y|}{\sqrt{\beta}} = \frac{|(x-y) +\sigma\beta(x+y)|}{\sqrt{\beta}} \leq \frac{|x-y|+\sigma|x-y|}{\sqrt{\beta}} = (1+\sigma)\sqrt{\alpha},
\end{align*}
we proved that
\begin{multline}\label{sqrtalpha}
K_{1}(x,y)\lesssim e^{-|x|^2+|y|^2}\bigg(\frac{|x+y|}{|x-y|}\bigg)^{n/2} \eta(x,y) \\ \times \sqrt{\alpha}\int_0^{1/\beta} \frac{(1-\sigma\beta)^n}{\sigma^{(n+3)/2}}(1+\sigma)e^{-\frac{1}{4}\alpha \varphi(\sigma)}\, \dd \sigma.
\end{multline}
Thus it remains only to prove that the quantity in~\eqref{sqrtalpha} is controlled by (a constant times) $\Phi$. Observe first that
\begin{multline*}
\sqrt{\alpha}\int_0^{1/\beta} \frac{(1-\sigma\beta)^n}{\sigma^{(n+3)/2}}(1+\sigma)e^{-\frac{1}{4}\alpha \varphi(\sigma)}\, \dd \sigma\\ \leq \sqrt{\alpha}\sup_{0<\sigma<1/\beta} \bigg( \frac{(1-\sigma \beta)^{n}}{\sigma^{n/2}}e^{-\frac{1}{4}\alpha \varphi(\sigma)}\bigg)^{1-\frac{1}{n}} \int_0^{1/\beta} (1+\sigma)\frac{1-\beta\sigma}{\sigma^2} e^{-\frac{1}{4n}\alpha\varphi(\sigma)}\, \dd \sigma
\\ \leq C \Phi(x,y)^{1-\frac{1}{n}} \sqrt{\alpha}\int_0^{1/\beta} (1+\sigma)\frac{1-\beta\sigma}{\sigma^2} e^{-\frac{1}{4n}\alpha\varphi(\sigma)}\, \dd \sigma.
\end{multline*}
The last inequality holds by~\eqref{supPhi}. We now separate the cases $\beta\geq 1$ and $\beta <1$. Observe that $\varphi$ is invertible in the intervals $(0,1)$ and $(1,\infty)$.

\textbf{1.} If $\beta\geq 1$, we have $1/\beta\leq 1$, and thus we can make the change of variables $\alpha\varphi(\sigma)=t$ in the integral, which gives
\begin{equation}\label{sigmameno}
\sigma = 1- \frac{\sqrt{t^2 +4\alpha t} -t}{2\alpha} \eqqcolon \sigma_-(t)\in (0,1]
\end{equation}
and hence
\begin{align*}
\sqrt{\alpha} \int_0^{1/\beta}\frac{(1-\sigma \beta)}{\sigma^2}(1+\sigma)e^{-\frac{1}{4n}\alpha \varphi(\sigma)}\, \dd \sigma &=\frac{1}{\sqrt{\alpha}}\int_{\alpha \varphi(1/\beta)}^\infty \frac{(1-\beta\sigma_-(t))}{1-\sigma_-(t)}e^{-\frac{1}{4n} t}\, \dd t \\& \leq \frac{1}{\sqrt{\alpha}} \int_0^\infty e^{-\frac{1}{4n} t}\, \dd t = C \Phi(x,y)^{\frac{1}{n}}.
\end{align*}
The first inequality holds since $\beta\geq 1$, and hence $1-\beta\sigma_-(t)\leq 1-\sigma_-(t)$.

\textbf{2.} If $\beta<1$ we split the integral
\[
\int_0^{1/\beta}\frac{(1-\sigma \beta)}{\sigma^2}(1+\sigma) e^{-\frac{1}{4n}\alpha \varphi(\sigma)}\, \dd \sigma = \int_0^{1}\dots \, \dd \sigma + \int_1^{1/\beta}\dots\, \dd \sigma,
\]
and we estimate the two integrals separately. We shall make the change of variables $\alpha \varphi(\sigma) = t$ in both cases. In the first integral we get as before $\sigma=\sigma_-(t)$ (see~\eqref{sigmameno}) and hence
\begin{align*}
\int_0^{1}\frac{(1-\sigma \beta)}{\sigma^2}(1+\sigma)e^{-\frac{1}{4n}\alpha\varphi(\sigma)}\, \dd \sigma =\frac{1}{\alpha}\int_0^\infty \frac{(1-\beta\sigma_-(t))}{1-\sigma_-(t)}e^{-\frac{1}{4n} t}\, \dd t.
\end{align*}
It is not hard to see that
\begin{align*}
1-\sigma_-(t) = \frac{\sqrt{t^2 +4\alpha t} -t}{2\alpha}\geq C\min \bigg(1,\frac{\sqrt{t}}{\sqrt{\alpha}}\bigg) \geq C\ \frac{\min (1,\sqrt{t})}{\sqrt{\alpha}},
\end{align*}
the first inequality since $\sqrt{1+z} -1\geq C \min (z,\sqrt{z})$, and the second inequality since $\alpha\geq 1/4$ by Lemma~\ref{lemmaprel}, (1).

Moreover
\[
1-\beta \sigma_-(t)= 1-\beta +\frac{\beta}{2\alpha}\bigg(\sqrt{t^2+4\alpha t} -t\bigg) \leq 1-\beta +\beta \frac{\sqrt{t}}{\sqrt{\alpha}} \leq 1-\beta + \frac{\sqrt{t}}{\sqrt{\alpha}},
\]
thanks to the inequality $\sqrt{a+b}\leq \sqrt{a}+ \sqrt{b}$. Therefore
\begin{align*}
\frac{1}{\alpha}\int_0^\infty \frac{(1-\beta\sigma_-(t))}{1-\sigma_-(t)^2}e^{-\frac{1}{4n} t}\, \dd t&\leq C \frac{1}{\sqrt{\alpha}}\int_0^\infty e^{-\frac{1}{4n} t}\max\bigg(1,\frac{1}{\sqrt{t}}\bigg)\bigg(1-\beta + \frac{\sqrt{t}}{\sqrt{\alpha}}\bigg)\, \dd t \\&\leq C\bigg(\frac{1-\beta}{\sqrt{\alpha}} +\frac{1}{\alpha}\bigg).
\end{align*}
We now consider the second integral. Here the change of variables $\alpha \varphi(\sigma)=t$ gives
\begin{equation}\label{sigmapiu}
\sigma =1+ \frac{\sqrt{t^2 +4\alpha t}+t}{2\alpha} \eqqcolon \sigma_+(t) \geq 1
\end{equation}
and
\begin{align*}
\int_1^{1/\beta}\frac{(1-\sigma \beta)}{\sigma^2}(1+\sigma) e^{-\frac{1}{4n}\alpha \varphi(\sigma)}\, \dd \sigma
&=\frac{1}{\alpha}\int_0^{\alpha \varphi(1/\beta)} \frac{(1-\beta\sigma_+(t))}{\sigma_+(t)-1}e^{-\frac{t}{4n} }\, \dd t\\
& \leq \frac{(1-\beta)}{\alpha}\int_0^\infty \frac{1}{\sigma_+(t)-1}e^{-\frac{t}{4n} }\, \dd t.
\end{align*}
Observe finally that
\[
\sigma_+(t)-1 = \frac{t+\sqrt{t^2+4\alpha t}}{2\alpha}\geq2 \frac{\sqrt{t}}{\sqrt{\alpha}}
\]
so that
\[
\frac{(1-\beta)}{\alpha}\int_0^\infty \frac{e^{-\frac{t}{4n} }}{\sigma_+(t)-1}\, \dd t \leq \frac{(1-\beta)}{\alpha} \int_0^\infty \frac{\sqrt{\alpha}}{\sqrt{t}} e^{-\frac{1}{4n} t}\, \dd t = C\frac{1-\beta}{\sqrt{\alpha}}
\]
which concludes the proof.
\end{proof}
\subsection{Proof of Theorem~\ref{teo:weaktype}, (i)}\label{subsec:wek:Impow}
The strategy of the proof is essentially the same as that for the Riesz transforms of the previous section. We shall then be very sketchy.
\begin{proposition}\label{impowloc}
For every $\lambda\geq0$ and every $u\in\R\setminus \{0\}$, $(\As+\lambda I)^{iu}_\loc$ is of weak type $(1,1)$ .
\end{proposition}
\begin{proof}
Again by Lemma~\ref{lemmalocal}
\[\abs{K_{(\As+\lambda I)^{iu}_\loc}(x,y)}\lesssim K_{2,0}(x,y)\lesssim |x-y|^{-n}\]
and
\begin{align*}
\abs{\nabla_{x} K_{(\As+\lambda I)^{iu}_\loc}(x,y)} + \abs{\nabla_{y} K_{(\As+\lambda I)^{iu}_\loc}(x,y)}&\lesssim K_{4,1}(x,y) + K_{2,0}(x,y)|x-y|^{-1}\\& \lesssim |x-y|^{-(n+1)}
\end{align*}
for every $(x,y)\in N_2$ such that $x\neq y$, $\lambda\geq0$. The conclusion follows again by~\cite[Theorem 3.2.8]{Salogni}.
\end{proof}
Before treating the global part, observe that by~\eqref{impowLeb}
\[
|K_{(\As+\lambda I)^{iu}}(x,y)|\lesssim \int_0^\infty \frac{e^{-(n+\lambda)t}}{(1-e^{-2t})^{(n+2)/2}} e^{-\frac{\abs{x-e^{-t}y}^2}{1-e^{-2t}}}\, \dd t \coloneqq K_\lambda'(x,y).
\]
\begin{proposition}\label{impowglobal}
Let $\lambda \geq 0$. Then, for every $(x,y)\in G$
\begin{equation}
K_\lambda'(x,y)\leq C\bar{K}(x,y).
\end{equation}
In particular, $(\As+\lambda I)^{iu}_\glob$ is of weak type $(1,1)$ for every $u\in \R\setminus \{0\}$ and $\lambda \geq 0$.
\end{proposition}

\begin{proof}
With the changes of variables $t=\tau(s)$ first, and then $s/\beta =\sigma$ in the integral defining $K'_\lambda$, we get as before
\begin{align*}
K'(x,y) \leq C e^{-|x|^2+|y|^2} \bigg(\frac{|x+y|}{|x-y|}\bigg)^{n/2} \eta(x,y) \int_0^{1/\beta}\frac{(1-\sigma \beta)^{n-1}}{\sigma^{n/2+1}}e^{-\frac{1}{4}\alpha \varphi(\sigma)}\, \dd \sigma
\end{align*}
and also
\[
\int_0^{1/\beta}\frac{(1-\sigma \beta)^{n-1}}{\sigma^{n/2+1}}e^{-\frac{1}{4}\alpha \varphi(\sigma)}\, \dd \sigma\leq C \Phi(x,y)^{1-\frac{1}{n}} \int_0^{1/\beta} \frac{1}{\sigma^{3/2}} e^{-\frac{1}{4n} \alpha\varphi(\sigma)}\, \dd \sigma,
\]
hence we only need to prove that
\[
\int_0^{1/\beta} \frac{1}{\sigma^{3/2}} e^{-\frac{1}{4n} \alpha\varphi(\sigma)}\, \dd \sigma\leq C \Phi(x,y)^{1/n}.
\] 
Observe first that 
\[
\int_0^{1/\beta} \frac{1}{\sigma^{3/2}} e^{-\frac{1}{4n} \alpha\varphi(\sigma)}\, \dd \sigma =\int_0^{\min(1,1/\beta)}\dots \, \dd \sigma + \int_{\min(1,1/\beta)}^{1/\beta}\dots\, \dd \sigma,
\] 
where the second integral is identically zero if $\beta\geq 1$. In both the integrals the function $\varphi$ is invertible, so that by the change of variables $\alpha \varphi(\sigma)=t$ we get as before
\[\int_0^{\min(1,1/\beta)} \frac{1}{\sigma^{3/2}} e^{-\frac{1}{4n} \alpha\varphi(\sigma)}\, \dd \sigma \leq \frac{C}{\alpha} \int_0^{\infty}\frac{\sqrt{\sigma_-(t)}}{1-\sigma_-(t)} e^{-\frac{t}{4n}}\, \dd t \]
while
\[\int_{\min(1,1/\beta)}^{1/\beta} \frac{1}{\sigma^{3/2}} e^{-\frac{1}{4n} \alpha\varphi(\sigma)}\, \dd \sigma \leq \frac{C}{\alpha} \int_0^{\infty}\frac{\sqrt{\sigma_+(t)}}{\sigma_+(t)-1} e^{-\frac{t}{4n}}\, \dd t,\]
where $\sigma_-(t)$ and $\sigma_+(t)$ are as in~\eqref{sigmameno} and~\eqref{sigmapiu}. Now observe that
\[
\frac{\sqrt{\sigma_\mp(t)}}{\pm1\mp\sigma_\mp(t)} = \frac{\sqrt{\alpha}}{\sqrt{t}}\sqrt{g_\mp(\alpha/t)}, \qquad g_\mp(z)=2 \frac{2z+1\mp\sqrt{1+4z}}{(\sqrt{1+4z}\mp 1)^2}.
\]
Since by elementary analysis both $g_-$ and $g_+$ are bounded on $[0,\infty)$, we obtain
\[
\int_0^{1/\beta} \frac{1}{\sigma^{3/2}} e^{-\frac{1}{4n} \alpha\varphi(\sigma)}\, \dd \sigma\leq \frac{C}{\sqrt{\alpha}} \int_0^\infty \frac{1}{\sqrt{t}}e^{-\frac{t}{4n}}\, \dd t = \frac{C}{\sqrt{\alpha}}\leq C \Phi(x,y)^{1/n},
\] 
and the proof is then complete.
\end{proof}

\section{Boundedness from $H^1(\gamma_{-1})$ to $L^1(\gamma_{-1})$}\label{BoundednessH1OUrovesciato}
In this section, we characterize the boundedness of $(\As +\lambda I)^{iu}$ and $\Rs_\lambda$ from the Hardy space $H^1(\gamma_{-1})$ to $L^1(\gamma_{-1})$, in terms of $\lambda\geq 0$. The main results are Theorems~\ref{ImpowH1} and~\ref{RieszH1} below, and their proof will occupy the remainder of this section.

\begin{theorem}\label{ImpowH1}
Let $\lambda\geq 0$ and $u\neq 0$. Then $(\As+\lambda I)^{iu}$ is bounded from $H^1(\gamma_{-1})$ to $L^1(\gamma_{-1})$ if and only if $\lambda>0$.
\end{theorem}

\begin{theorem}\label{RieszH1}
Let $\lambda\geq 0$ and $j=1,\dots,n$.
\begin{itemize}
\item[\emph{(i)}] If $n=1$, $(\Rs_\lambda)_j$ is bounded from $H^1(\gamma_{-1})$ to $L^1(\gamma_{-1})$ if and only if $\lambda>1$;
\item[\emph{(ii)}] if $n\geq 2$, $(\Rs_\lambda)_j$ is not bounded from $H^1(\gamma_{-1})$ to $L^1(\gamma_{-1})$.
\end{itemize}
\end{theorem}

\subsection{Proof of Theorem~\ref{ImpowH1}}
We first need a preliminary lemma, inspired by~\cite[Lemma 7.1]{MM}.

\begin{lemma}\label{lemmatech}
For every ball $B$ and $y\in B$, let
\[r_{B,y}=r_B/(4|y|).\]
\begin{itemize}
\item[\emph{1.}] If $r_{B,y}\geq 1$, then $\abs{rx-y}\geq |x-c_B|/8$ for every $r\in (3/4,1)$ and $x\in (2B)^c$;
\item[\emph{2.}] If $r_{B,y}<1$, then $\abs{rx-y}\geq |x-c_B|/8$ for every $r\in (1- r_{B,y},1)$ and $x\in (2B)^c$.
\item[\emph{3.}] For every $c>0$, there exists $C>0$ such that
\[(1-r^2)^{-n/2}\int_{(2B)^c} e^{-c\frac{|x-c_B|^2}{1-r^2}}\, \dd x \leq C \varphi(r_B/\sqrt{1-r^2})\]
for every $r\in [0,1]$ and $B\in \mathcal{B}_1$, where $\varphi(s)=(1+s)^{n-2} e^{-s^2}$.
\end{itemize}
\end{lemma}
\begin{proof}
We begin by 1. If $r\in (3/4,1)$, $x\in (2B)^c$ and since $|y|\leq r_B/4$ by assumption 
\begin{align*}
|rx-y| &\geq r|x-c_B| - r|c_B-y|-(1-r)|y|
\\& \geq \frac{3}{4}|x-c_B| -r_B -\frac{1}{16}r_B \geq \frac{1}{8}|x-c_B| + \bigg[\frac{5}{8}|x-c_B| -\frac{17}{16}r_B\bigg]\geq \frac{1}{8}|x-c_B| 
\end{align*}
since
\[\frac{5}{8}|x-c_B| -\frac{17}{16} r_B \geq \frac{5}{4}r_B -\frac{17}{16}r_B\geq 0. \]
The point 2.\ can be proved analogously, since by assumption $(1-r)\leq r_B/(4|y|)$ and then
\begin{align*}
|rx-y| &\geq r|x-c_B| - r|c_B-y|-(1-r)|y|\geq \frac{3}{4}|x-c_B| -r_B -\frac{1}{4}r_B \geq \frac{1}{8}|x-c_B|
\end{align*}
as above. The point 3.\ is~\cite[Lemma 7.1, (iii)]{MM}.
\end{proof}

\begin{proof}[Proof of Theorem~\ref{ImpowH1}]
We begin by proving that $\As^{iu}$ is unbounded from $H^1(\gamma_{-1})$ to $L^1(\gamma_{-1})$. Let
\[E_+\coloneqq B(0,1) \cap \{ y_n\geq 0\}, \qquad E_-\coloneqq B(0,1) \cap \{ y_n<0\}\]
be the upper and lower hemispheres or radius 1, respectively, and consider the atom
\begin{equation}\label{atomexpl}
a(y) = \gamma_{-1}(B)^{-1}(\chi_{E_+}(y) - \chi_{E_-}(y)).
\end{equation}
If $e_n=(0,\dots,0,1)$
\begin{equation}\label{aPsien}
(a,\Psi_{0,e_n}) =\int_0^\infty e^{-t}t^{(n-2)/2} \bigg(\int_{E_+}\bigg[ e^{2(e_n,y) \sqrt{t}} - e^{-2(e_n,y) \sqrt{t}}\bigg]\, \dd y \bigg) \, \dd t>0
\end{equation}
since the function $s\mapsto e^{s}-e^{-s}$ is positive for every $s>0$. Thus, $\As^{iu} a \notin L^1(\gamma_{-1})$ by Lemma~\ref{lemmaunbounded0}.

Let now $\lambda>0$. We prove that $(\As +\lambda I)^{iu}$ is bounded $H^1(\gamma_{-1}) \to L^1(\gamma_{-1})$ by verifying the local Hormander-type condition (cf.~\cite[Theorem 8.2 and Remark 8.3]{CMM})
\begin{equation}\label{localHormImg}
\sup_{B\in \mathcal{B}_1} r_B \sup_{y\in B} \int_{(2B)^c}\abs{\nabla_y k_{(\As+\lambda I)^{iu}}(x,y)}\gamma_{-1}(x)\,\dd x<\infty.
\end{equation}
Up to a constant factor, by~\eqref{impowkernel} the integral in~\eqref{localHormImg} is
\begin{align}\label{inf}
I(y,B)&=\int_{(2B)^c} \bigg|\int_0^1 \frac{r^{n+\lambda-1}(-\log r)^{-iu-1}}{(1-r^2)^{n/2}} \nabla_y e^{-|\psi(r,x,y)|^2} \,\dd r\bigg|\, \dd x \nonumber \\&\leq \int_0^1 \frac{r^{n+\lambda-1}(-\log r)^{-1}}{(1-r^2)^{(n+1)/2}} \int_{(2B)^c} 2\frac{\abs*{rx-y}}{\sqrt{1-r^2}} e^{-\frac{|rx-y|^2}{1-r^2}}\, \dd x \, \dd r.
\end{align}
By using the inequality $s \exp(-s^2)\leq c \exp(-s^2/2)$ for every $s\geq 0$,
\begin{align*}
I(y,B) &\lesssim \int_0^1 \frac{r^{n+\lambda-1}(-\log r)^{-1}}{(1-r^2)^{(n+1)/2}} \int_{(2B)^c} e^{-\frac{|rx-y|^2}{2(1-r^2)}}\, \dd x \, \dd r
\\& = \bigg(\int_0^{\frac{3}{4}} + \int_{\frac{3}{4}}^1\bigg)\frac{r^{n+\lambda-1}(-\log r)^{-1}}{(1-r^2)^{(n+1)/2}} \int_{(2B)^c} e^{-\frac{|rx-y|^2}{2(1-r^2)}}\, \dd x \, \dd r =I_1(y,B) + I_2(y,B).
\end{align*}
As for $I_1(y,B)$,
\begin{align*}
I_1(y,B) &\leq C \int_0^{3/4} \frac{r^{n+\lambda-1}}{(-\log r)} \int_{(2B)^c} e^{-c |rx-y|^2}\, \dd x \, \dd r
\\& \leq C \int_0^{3/4} \frac{r^{n+\lambda-1}}{(-\log r)} \frac{1}{r^n}\int_{\R^n} e^{-|v|^2}\, \dd v \, \dd r =C \int_0^{3/4} \frac{r^{\lambda-1}}{(-\log r)}\leq C
\end{align*}
for every $y\in B$ and every $\lambda >0$.

As for $I_2(y,B)$, observe that there exists $C>0$ such that
\[\frac{1}{(1-r^2)^{(n+1)/2}(- \log r)}\leq \frac{C}{(1-r^2)^{(n+3)/2}}\]
for every $r\in (3/4,1)$. By means of Lemma~\ref{lemmatech} we can then argue exactly as in~\cite[p.\ 310]{MM}. Thus, we get
\[ 
I_2(y,B)\leq \frac{C}{r_B}\qquad \mbox{if $r_{B,y}\geq 1$}.
\]
If $r_{B,y}<1$, we split
\[(3/4,1) = (3/4, 1-r_{B,y}) \cup (1-r_{B,y},1)\]
meaning that $(3/4, 1-r_{B,y})$ is the empty set if $3/4>1-r_{B,y}$, and thus one of the integrals below is identically zero. We split accordingly
\[
I_2(y,B)= I_2^1(y,B) + I_2^2(y,B)
\]
and again as~\cite[p.\ 310]{MM}, we get
\[ 
I_2^j(y,B)\leq \frac{C}{r_B}\qquad \mbox{if $r_{B,y}<1$, for $j=1,2$}.
\]
Now~\eqref{localHormImg} follows easily.
\end{proof}

\subsection{Proof of Theorem~\ref{RieszH1}}

\begin{proof}[Proof of Theorem~\ref{RieszH1}]
We first prove that if $\lambda \in [0,1]$, then $(\Rs_\lambda)_j \colon H^1(\gamma_{-1}) \not \to L^1(\gamma_{-1})$ for every $n\geq 1$. To do this, let $a$ be the $H^1$-atom defined in~\eqref{atomexpl}, and observe that~\eqref{aPsien} together with Lemma~\ref{lemmaunbounded}, (i) imply that $(\Rs_\lambda)_j a \notin L^1(\gamma_{-1})$. This proves both the ``only if'' part of (i), and (ii) for $\lambda\in [0,1]$.

\smallskip

To complete the proof of (i), it remains to prove that $\Rs_\lambda$ is bounded from $H^1(\gamma_{-1})$ to $L^1(\gamma_{-1})$ when $n=1$ and $\lambda>1$. We do this by verifying the local Hormander-type condition (cf.~\cite[Theorem 8.2 and Remark 8.3]{CMM})
\begin{equation}\label{localHormRiesz}
\sup_{B\in \mathcal{B}_1} r_B \sup_{y\in B} \int_{(2B)^c}\abs{\partial_y k_{\Rs_\lambda}(x,y)}\gamma_{-1}(x)\,\dd x<\infty.
\end{equation}
Up to a constant factor, by~\eqref{Rieszkernel} the integral in~\eqref{localHormRiesz} is
\begin{equation}\label{inf}
\int_{(2B)^c} \bigg|\int_0^1 \frac{r^{\lambda}}{\sqrt{- \log r} (1-r^2)} \partial_y \bigg[\phi(r,x,y) e^{-\psi(r,x,y)^2}\bigg] \,\dd r\bigg|\, \dd x.
\end{equation}
We first focus on the inner integral of~\eqref{inf}, which is
\[\int_0^1 \frac{r^{\lambda}(\partial_y \phi)e^{-\psi(r,x,y)^2} }{\sqrt{- \log r} (1-r^2)} \, \dd r + \int_0^1 \frac{-2r^\lambda \psi \phi (\partial_y \psi) e^{-\psi^2}}{\sqrt{- \log r} (1-r^2)} \, \dd r.\]
The first term is
\[\int_0^1 e^{-\psi(r,x,y)^2} \frac{r^{\lambda+1} }{(1-r^2)^{3/2} \sqrt{-\log r}}\, \dd r\]
Since $-\phi= (1-r^2)\partial_r \psi$ and $2\psi e^{-\psi^2}(\partial_r \psi) = -(\partial_r e^{-\psi^2})$, the second term is
\begin{equation}\label{infinizio}
\begin{split}
\int_0^1 \frac{2r^\lambda \psi (\partial_r \psi) e^{-\psi^2} (\partial_y \psi)}{\sqrt{-\log r}}\, \dd r = - \int_0^1 \frac{-r^\lambda }{\sqrt{1-r^2}\sqrt{-\log r}}(\partial_r e^{-\psi^2}) \,\dd r. 
\end{split}
\end{equation}
Integrating by parts, the boundary terms vanish and~\eqref{infinizio} equals
\[
\begin{split}
\int_0^1 e^{-\psi^2}\bigg( \frac{-\lambda r^{\lambda-1}}{\sqrt{1-r^2}\sqrt{-\log r}} - \frac{r^{\lambda-1}}{2\sqrt{1-r^2} (-\log r)^{3/2}} -\frac{r^{\lambda+1}}{(1-r^2)^{3/2} \sqrt{-\log r}}\bigg) \, \dd r.
\end{split}
\]
Therefore, two terms cancel out in the inner integral of~\eqref{inf}, which then equals
\begin{align*}
\int_{(2B)^c} \bigg|\int_0^1 e^{-\psi^2}\bigg( &\frac{-\lambda r^{\lambda-1}}{\sqrt{1-r^2}\sqrt{-\log r}} - \frac{r^{\lambda-1}}{2\sqrt{1-r^2} (-\log r)^{3/2}}\bigg) \,\dd r\bigg|\, \dd x 
\\& \leq C \int_{(2B)^c} \int_0^1 \frac{r^{\lambda-1}e^{-\psi^2}}{\sqrt{1-r^2} }\bigg(\frac{1}{\sqrt{-\log r}} + \frac{1}{(-\log r)^{3/2}}\bigg)\, \dd r\, \dd x \\&= J_1(y,B) + J_2(y,B).
\end{align*}
For each $\ell$ we split each of the integrals $J_\ell$ into $J_\ell^1$ and $J_\ell^2$ according to the splitting $(0,1)= (0,3/4)\cup (3/4,1)$. First observe that
\begin{align*}
J_\ell^1(y,B) &\leq C \int_0^{3/4} \frac{r^{\lambda-1}}{(-\log r)^{1/2}} \int_{(2B)^c} e^{-c(rx-y)^2}\, \dd x\, \dd r \leq C \int_0^{3/4} \frac{r^{\lambda-2}}{(-\log r)^{1/2}} \, \dd r\leq C
\end{align*}
for every $y\in B$ and every $\lambda>1$. As for the remaining integrals,
\[J_\ell^2(y,B)\leq C \int_{3/4}^1 \frac{1}{(1-r^2)^2}\int_{(2B)^c}e^{-\frac{(rx-y)^2}{1-r^2}}\, \dd x\, \dd r\]
and thus one can argue as above (\emph{i.e.}, as~\cite[p.\ 310]{MM} together with Lemma~\ref{lemmatech} when $n=1$) to get
\[J_\ell^2(y,B)\leq \frac{C}{r_B}\]
for every $\ell=1,2$. This completes the proof of (i).\par

\smallskip

(ii) Let now $\lambda\geq 0$ and $n\geq 2$. Inspired by~\cite{MMS}, we prove that $(\Rs_\lambda)_j$ is unbounded on admissible cubes at scale 1 when $j=1$. The statement for $j=1$ then follows by Remark~\ref{rem:cubic}, and the proofs for $j>1$ are analogous. Observe, however, that the statement for $\lambda\in [0,1]$ and every $j=1,\dots,n$ follows by the first part of the proof.

Consider the family of cubes $Q=Q(\xi)$ centred in $(\xi,0,\dots,0)$ and of sidelength $2/\xi$, for $\xi$ large. Let 
\[Q_+ = Q \cap \{y_2\geq 0\}, \quad Q_- = Q \cap \{y_2 < 0\} \]
and define
\[ a= - \frac{1}{\gamma_{-1}(Q)}(\chi_{Q_+} - \chi_{Q_-}).\]
Since $a$ is an atom supported in an admissible cube for every $\xi$ large enough, its atomic norm is bounded independently of $\xi$. We shall prove that $\|(\Rs_\lambda)_1 a\|_{L^1(\gamma_{-1})} \to \infty$ as $\xi \to \infty$.

First of all, observe that by~\eqref{Rieszkernel}
\[(\Rs_\lambda)_1 a(x) e^{|x|^2} = \frac{2 \pi^{-\frac{n+1}{2}}}{\gamma_{-1}(Q)} \int_{Q_+} \int_0^1 \frac{r^{n+\lambda-1}(x_1-ry_1)e^{-|\psi(r,x,y)|^2}}{(1-r^2)^{(n+2)/2}\sqrt{- \log r}} (1-e^{-\tau(r,x_2,y_2)}) \, \dd r \,\dd\gamma_{-1}(y)\]
where
\[\tau(r,x_2,y_2)=\frac{4rx_2y_2}{1-r^2}.\] 
Define now the function
\[\nu(\xi,x_1) \coloneqq \sqrt{\frac{x_1 -\xi}{x_1}}, \qquad x\in Q\]
and the set
\begin{multline*}
A_Q \coloneqq \bigg\{x\in \R^n \colon x_1 \in \bigg(\xi +\frac{4}{\xi}, \xi +1\bigg), \; x_2 \in \bigg[\frac{\nu(\xi,x_1)}{2},\nu(\xi,x_1)\bigg], \\ |x_k|\leq \nu(\xi,x_1), \, 3\leq k\leq n\bigg\}. 
\end{multline*}
If $x\in A_Q$, $y\in Q_+$ and $r\in (0,1)$, then $x_1 -ry_1 > 0$ and $\tau(r,x_2,y_2)\geq 0$, so that $1-e^{-\tau(r,x_2,y_2)}\geq 0$. Thus the integral defining $(\Rs_\lambda)_1 a(x) e^{|x|^2}$ is positive. Therefore, if $x\in A_Q$ we can bound $(\Rs_\lambda)_1 a(x)e^{|x|^2}$ from below by restricting the integral over $(0,1)$ to the integral on the subset
\[I(x_1) \coloneqq \bigg\{r\in (0,1) \colon \bigg|r-\frac{\xi}{x_1}\bigg|< \frac{\nu(\xi,x_1)}{2\xi}\bigg\},\]
whose measure is $|I(x_1)|= \nu(\xi,x_1)/\xi$. Since, for every $r\in I(x_1)$,
\[
\begin{split}
r > \frac{\xi}{x_1} -\frac{\nu(\xi,x_1)}{2\xi}& 
%> \frac{\xi}{\xi+1} - \frac{1}{2\xi} \sqrt{\frac{x_1-\xi}{x_1}}\\&
 > \frac{\xi}{\xi+1} - \frac{1}{2}\frac{1}{\xi \sqrt{\xi}} \xrightarrow{\xi \to \infty}{ 1},
\end{split}
\]
we can suppose $\xi$ large enough to imply $I(x_1) \subset \left(\frac{1}{2},1\right)$. With this choice $-\log r \leq C(1-r^2)$ for every $r\in I(x_1)$, so that if $x\in A_Q$ then
\[
\begin{split}
(\Rs_\lambda)_1 a(x) e^{|x|^2} \geq \frac{c}{\gamma_{-1}(Q)} \int_{Q_+} \int_{I(x_1)} \frac{(x_1-ry_1)e^{-|\psi(r,x,y)|^2}}{(1-r^2)^{(n+3)/2}} (1-e^{-\tau(r,x_2,y_2)}) \, \dd r\, \dd\gamma_{-1}(y)
\end{split}
\]
which yields
\begin{multline*}
\|(\Rs_\lambda)_1 a\|_{L^1(\gamma_{-1})} \geq \frac{c}{\gamma_{-1}(Q)} \int_{A_Q}\int_{Q_+} \int_{I(x_1)} \frac{(x_1-ry_1)}{(1-r^2)^{(n+3)/2}} e^{-|\psi(r,x,y)|^2}\\ \times (1-e^{-\tau(r,x_2,y_2)}) \, \dd r \,\dd\gamma_{-1}(y)\, \dd x.
\end{multline*}
Assume for a moment that for $x\in A_Q$, $y\in Q_+$ and $r \in I(x_1)$
\begin{equation}\label{eq1}
1-r^2 \approx \nu(\xi,x_1)^2,
\end{equation}
\begin{equation}\label{eq2}
x_1 -ry_1 \geq c(x_1-\xi),
\end{equation}
\begin{equation}\label{eq3}
e^{-|\psi(r,x,y)|^2} \geq c,
\end{equation}
\begin{equation}\label{eq4}
1-e^{-\tau(r,x_2,y_2)} \approx \frac{y_2}{\nu(\xi,x_1)},
\end{equation}
\begin{equation}\label{eq5}
\frac{1}{\gamma_{-1}(Q)}\int_{Q_+}y_2 \gamma_{-1}(y)\,\dd y \geq \frac{c}{\xi}.
\end{equation}
Then,
\[
\begin{split}
\|(\Rs_\lambda)_1 a\|_{L^1(\gamma_{-1})} &\geq \frac{c}{\gamma_{-1}(Q)} \int_{A_Q} \int_{Q_+} \frac{1}{\nu(\xi,x_1)^{n+3}} \frac{y_2}{\nu(\xi,x_1)} (x_1-\xi) |I(x_1)|\gamma_{-1}(y)\, \dd y \, \dd x\\& \geq c\int_{A_Q} \frac{1}{\nu(\xi,x_1)^{n+1}} \frac{1}{\xi}\, \dd x\\&= \frac{c}{\xi} \int_{\xi + 4/\xi}^{\xi +1} \frac{x_1}{x_1-\xi}\, \dd x_1 \geq c \log \xi,
\end{split}
\]
and the unboundedness of $(\Rs_\lambda)_1$ follows. It remains then to prove~\eqref{eq1}-\eqref{eq5}. From now on, we assume that $x\in A_Q$, $y\in Q_+$ and $r \in I(x_1)$.

We start by observing that, since \[\frac{4}{x_1}<\frac{4}{\xi}< x_1 -\xi,\]
then by taking the geometric mean of the first and third quantities we obtain
\begin{equation}\label{noneps}
\frac{4}{x_1}< 2\nu(\xi,x_1) < x_1-\xi.
\end{equation}
For $\xi$ large enough, this yields
\begin{equation}\label{eps}
\frac{3}{\xi}< 2\nu(\xi,x_1) < x_1-\xi.
\end{equation}
We now prove~\eqref{eq1}. Since $r >\frac{1}{2}$, we get
\[1-r^2 \approx 1-r = \nu(\xi,x_1)^2 - \bigg(r- \frac{\xi}{x_1}\bigg)\]
where the equality holds by definition of $\nu$. By the definition of $I(x_1)$ and~\eqref{eps}, we get
\[\bigg| r- \frac{\xi}{x_1}\bigg| < \frac{\nu(\xi,x_1)}{2\xi}< \frac{\nu(\xi,x_1)^2}{3},\]
which gives~\eqref{eq1}. 

As for~\eqref{eq2}, observe that the definition of $I(x_1)$ and~\eqref{eps} also imply
\[r < \frac{\nu(\xi,x_1)}{2\xi} + \frac{\xi}{x_1} < \frac{x_1 - \xi}{4 \xi} + \frac{\xi}{x_1}\]
and since $0\leq y_1 < \xi + \frac{1}{\xi}$,
\[
\begin{split}
x_1 &-r y_1 > x_1 - \bigg( \frac{x_1 - \xi}{4 \xi} + \frac{\xi}{x_1}\bigg) \bigg( \xi + \frac{1}{\xi}\bigg)\\&= \bigg(x_1 - \frac{\xi^2}{x_1}\bigg) - \frac{x_1 - \xi}{4} - \frac{1}{\xi}\bigg(\frac{x_1 - \xi}{4\xi} + \frac{\xi}{x_1}\bigg)\\&> (x_1 - \xi) - \frac{x_1 - \xi}{4} - \frac{1}{\xi}\bigg(\frac{x_1 - \xi}{4\xi} + \frac{\xi}{x_1}\bigg)= \frac{3}{4}(x_1 - \xi)-\frac{x_1 - \xi}{4\xi^2} - \frac{1}{x_1} \geq \frac{1}{4}(x_1 -\xi),
\end{split}
\]
where we used~\eqref{noneps} in the last inequality. Thus~\eqref{eq2} is proved.

To prove~\eqref{eq3}, observe that
\[
\begin{split}
|rx_1 -y_1| &\leq \bigg| x_1\bigg(r- \frac{\xi}{x_1}\bigg)\bigg| + | \xi -y_1 |\\&\leq (\xi +1)\frac{\nu(\xi,x_1)}{2\xi} + \frac{1}{\xi}\leq c\bigg[ \frac{\nu(\xi,x_1)}{2} + \frac{1}{x_1}\bigg] \leq c \, \nu(\xi,x_1) \leq c \sqrt{1-r^2}
\end{split}
\]
where the last two inequalities are a consequence of~\eqref{noneps} and~\eqref{eq1} respectively. If $k\geq 2$, by~\eqref{eps} and the definition of $A_Q$
\[
\begin{split}
|rx_k -y_k| &\leq |x_k| + |y_k| \leq \nu(\xi,x_1) + \frac{1}{\xi} \leq c \nu(\xi,x_1) \leq c \sqrt{1-r^2},
\end{split}
\]
where the last inequality holds again by~\eqref{eq1}. This shows that $|\psi|^2 \leq c$, from which~\eqref{eq3} follows. 

For~\eqref{eq4}, observe that
\begin{equation}\label{tau}
\tau = \frac{4rx_2 y_2}{1-r^2} \approx \frac{4r \nu(\xi,x_1) y_2}{\nu(\xi,x_1)^2} \approx \frac{y_2}{\nu(\xi,x_1)}
\end{equation}
since $x_2 \approx \nu(\xi,x_1)$ and by~\eqref{eq1}, and that
\[0\leq \tau \leq c\frac{y_2}{\nu(\xi,x_1)} \leq c \frac{1}{\xi} \sqrt{\frac{x_1}{x_1 - \xi}} \leq c \sqrt{\frac{\xi+1}{\xi}} \leq c\]
by~\eqref{tau} and~\eqref{eps}, if $\xi$ is sufficiently large. Therefore
\[
1-e^{-\tau} \approx \tau \approx \frac{y_2}{\nu(\xi,x_1)}
\]
which proves~\eqref{eq4}.

Finally, we prove~\eqref{eq5}. Since for every positive and sufficiently small $u$ one has $e^{u^2}\leq 1+u$, while $e^u -1 \geq u$ for every $u\geq 0$, if $\xi$ is large enough
\[
\frac{1}{\gamma_{-1}(Q)}\int_{Q_+}y_2 \gamma_{-1}(y)\, \dd y = c \frac{\int_0^{1/\xi}y_2 e^{y_2^2} \, \dd y_2}{\int_0^{1/\xi}e^{y_2^2} \, \dd y_2} \geq c \frac{e^{1/\xi^2} -1}{\int_0^{1/\xi}e^{y_2^2} \, \dd y_2} \geq \frac{c}{\xi}
\]
which is~\eqref{eq5}. This completes the proof of (ii) and of the theorem.
\end{proof}

\section{Boundedness from $X^1_\mu(\gamma_{-1})$ to $L^1(\gamma_{-1})$}\label{BoundednessX1OUrovesciato}
Theorems~\ref{teo:X1Impow} and~\ref{teo:X1Riesz} below characterize the boundedness of $(\As +\lambda I)^{iu}$ and $(\Rs_\lambda)_j$, $\lambda\geq 0$, from the Hardy spaces $X^1_\mu(\gamma_{-1})$, $\mu\geq 0$, to $L^1(\gamma_{-1})$, in terms of $\lambda$ and $\mu$. Their proof will occupy the remainder of the paper.

\begin{theorem}\label{teo:X1Impow}
Let $\lambda,\mu \geq 0$ and $u\neq 0$.
\begin{itemize}
\item[\emph{(i)}] $\As^{iu}$ is bounded from $X^1_\mu(\gamma_{-1})$ to $L^1(\gamma_{-1})$ if and only if $\mu =0$;
\item[\emph{(ii)}] if $\lambda>0$, then $(\As+\lambda I)^{iu}$ is bounded from $X^1_\mu(\gamma_{-1})$ to $L^1(\gamma_{-1})$ for every $\mu\geq 0$.
\end{itemize}
\end{theorem}

\begin{theorem}\label{teo:X1Riesz}
Let $\lambda,\mu \geq 0$ and $j=1,\dots, n$.
\begin{itemize}
\item[\emph{(i)}] If $\lambda \in [0,1)$, then $(\Rs_\lambda)_j$ is unbounded from $X^1_\mu(\gamma_{-1})$ to $L^1(\gamma_{-1})$ for every $\mu \geq 0$;
\item[\emph{(ii)}] $(\Rs_1)_j$ is bounded from $X^1_\mu(\gamma_{-1}) $ to $L^1(\gamma_{-1})$ if and only if $\mu=1$;
\item[\emph{(iii)}] if $\lambda>1$, then $(\Rs_\lambda)_j$ is bounded from $X^1_\mu(\gamma_{-1})$ to $L^1(\gamma_{-1})$ for every $\mu\geq 0$.
\end{itemize}
\end{theorem}
\subsection{Proof of Theorem~\ref{teo:X1Impow}}
\begin{lemma}\label{almostTakeda}
For every $t\in (0,r_B^2]$ and for all $B\in \mathcal{B}_1$
\begin{equation*}
\frac{1}{\gamma_{-1}(B)}\int_B (e^{-t\As}1_{(2B)^c})^2\, \dd \gamma_{-1} \leq C e^{-cr_B^2/t}.
\end{equation*}
\end{lemma}

\begin{proof}
Since the operator $-\As$ is the weighted Laplacian on the weighted manifold $(\R^n,d_{Euc}, \gamma_{-1})$ and $\gamma_{-1}(2B\setminus B) \approx \gamma_{-1}(B)$ for all admissible balls $B$ by Lemma~\ref{adm}, the estimate is a consequence of Takeda's inequality~\cite[Theorem 12.9]{Grigor'yan}.
\end{proof}

\begin{proof}[Proof of Theorem~\ref{teo:X1Impow}]
(i) We begin by proving that $\As^{iu} \colon X^1_\mu(\gamma_{-1})\not \to L^1(\gamma_{-1})$ if $\mu>0$. Let $B=B(0,1)$, $\sigma_0\in S^{n-1}$ and choose $\psi_0 \in C_c^\infty(B)$ with integral zero with respect to $\dd \gamma_{-1}$ which is not orthogonal to $\Psi_{0,\sigma_0}$ in $L^2(B,\gamma_{-1})$. Such a $\psi_0$ exists by Lemma~\ref{Psi}, (3). Let $a= (\As+\mu I)\psi_0$, which is a multiple of an $X^1_\mu$-atom by Corollary~\ref{corlnotm}. Then
\begin{align*}
(\Psi_{0,\sigma_0}, a)_{L^2(B,\gamma_{-1})}=((\As +\mu I)\Psi_{0,\sigma_0}, \psi_0)_{L^2(B,\gamma_{-1})} = \mu (\Psi_{0,\sigma_0}, \psi_0)_{L^2(B,\gamma_{-1})}\neq 0.
\end{align*}
Thus, $\As^{iu} a \notin L^1(\gamma_{-1})$ by Lemma~\ref{lemmaunbounded0}.\par
To prove that $\As^{iu}$ is bounded $X^1_0(\gamma_{-1})\to L^1(\gamma_{-1})$, we follow the same line as \cite[Theorem 4.1]{MMV}, though with some differences, to prove that
\[
\sup \{ \|\As^{iu} a\|_1 \colon \, \mbox{$a$ is an $X^1_0$-atom}\}<\infty.
\]
Since $\As^{iu}$ is of weak type $(1,1)$ by Theorem~\ref{teo:weaktype}, (i), this implies the boundedness $X^1_0(\gamma_{-1})\to L^1(\gamma_{-1})$ by a classical argument~\cite[p.\ 95]{Grafakos}.\par

Let $a$ be an $X^1_0$-atom supported in an admissible ball $B$. We split
\[\|\As^{iu} a\|_1 = \|1_{2B} \As^{iu}a\|_1 + \|1_{(2B)^c}\As a\|_1.\]
By Schwarz's inequality, the size condition for $a$ and the spectral theorem
\[ \|1_{2B}\As^{iu}a \|_1\leq \gamma_{-1}(2B)^{1/2} \| \As^{iu}\|_2 \|a\|_2 \leq \bigg(\frac{\gamma_{-1}(2B)}{\gamma_{-1}(B)}\bigg)^{1/2}\]
which is bounded independently of $B$ since $\gamma_{-1}$ is locally doubling. 

As for the other term, we can write $\As^{iu}a  = \As^{iu+1} \As^{-1}a$.
Thus, by Schwarz's inequality and Theorem~\ref{support:pres}
\begin{align*}
\|1_{(2B)^c} \As^{iu}a\|_1 & \leq \|\As^{-1}a\|_2 \bigg[\int_B \bigg(\int_{(2B)^c}|k_{\As^{iu+1}}(x,y)|\,\dd\gamma_{-1}(x)\bigg)^2 \, \dd \gamma_{-1}(y)\bigg]^{1/2}\\&\leq\,C\, r_B^2 \bigg[\frac{1}{\gamma_{-1}(B)}\int_B \bigg(\int_{(2B)^c}|k_{\As^{iu+1}}(x,y)|\,\dd\gamma_{-1}(x)\bigg)^2 \, \dd \gamma_{-1}(y)\bigg]^{1/2}.
\end{align*}
It remains then to show that
\begin{equation}\label{lerB-2}
\frac{1}{\gamma_{-1}(B)}\int_B \bigg(\int_{(2B)^c}|k_{\As^{iu+1}}(x,y)|\,\dd\gamma_{-1}(x)\bigg)^2 \, \dd \gamma_{-1}(y)\bigg]^{1/2}\leq C r_B^{-2}
\end{equation}
for some $C$ independent of the ball $B$. By the subordination formula,
\[k_{\As^{iu+1}}(x,y) = c(u)\int_0^\infty t^{-iu-2} h_t(x,y)\, \dd t,\qquad\forall x\in B, y\in (2B)^c.\]
We split the integral on the right hand side as the sum of integrals over $(0,r_B^2]$ and $(r_B^2,\infty)$. Then
\begin{align*}
\int_{(2B)^c} \bigg|\int_{r_B^2}^\infty t^{-iu-2} h_t(x,y)\,\dd t\bigg| \, \dd \gamma_{-1}(x)&\leq \int_{r_B^2}^\infty t^{-2}\, \dd t \int_{(2B)^c}h_t(x,y)\, \dd \gamma_{-1}(x) \leq r_{B}^{-2}
\end{align*}
by the contractivity of $e^{-t\As}$ on $L^\infty$. Thus\[\bigg[\frac{1}{\gamma_{-1}(B)}\int_B \bigg(\int_{(2B)^c}\bigg|\int_{r_B^2}^\infty t^{-iu-2} h_t(x,y)\, \dd t\bigg|\,\dd\gamma_{-1}(x)\bigg)^2 \, \dd \gamma_{-1}(y)\bigg]^{1/2}\leq C \,r_B^{-2}.
\]
As for the integral over $(0,r_B^2]$, we observe that
\begin{align*}
\int_{(2B)^c} \left|\int_0^{r_B^2} t^{-iu-2} h_t(x,y)\, \dd t\right| \,\dd\gamma_{-1}(x) 
&\le \int_0^{r_B^2}t^{-2} \int_{(2B)^c} h_t(x,y) \,\dd\gamma_{-1}(x) \,\dd t \\
&= \int_0^{r_B^2}t^{-2} e^{-t\As}1_{(2B)^c}(y)\ \dd t,
\end{align*}
where, in the last line, we have used the symmetry of $h_t(x,y)$. Thus
\begin{align*}
\bigg[\frac{1}{\gamma_{-1}(B)}\int_B&\left(\int_{(2B)^c} \left|\int_0^{r_B^2} t^{-iu-2} h_t(x,y) \,\dd t \right|\,\dd\gamma_{-1}(x)\right)^2\, \dd\gamma_{-1}(y) \bigg]^{1/2}\\
&\le \bigg[\frac{1}{\gamma_{-1}(B)}\int_B\left(\int_0^{r_B^2} t^{-2} e^{-t\As}1_{(2B)^c}(y) \dd t\right)^2 \,\dd\gamma_{-1}(y)\bigg]^{1/2}\\
&\le \int_0^{r_B^2} t^{-2} \bigg[\frac{1}{\gamma_{-1}(B)}\int_B\left( e^{-t\As}1_{(2B)^c}(y)\right)^2 \,\dd\gamma_{-1}(y)\bigg]^{1/2}\ \dd t.
\end{align*}
The last line follows by Minkowski's integral inequality. By Lemma~\ref{almostTakeda} the last integral is bounded by 
\[\int_{0}^{r_B^2}e^{-cr_B^2/(2t)}t^{-2}\, \dd t = \frac{1}{r_B^2}\int_0^1 e^{-c/(2r)}r^{-2}\, \dd r\leq Cr_{B}^{-2}.\]
This concludes the proof of (i).\par

\smallskip

To prove (ii), observe that if $\lambda>0$ then $(\As+\lambda I)^{iu}$ is bounded $H^1(\gamma_{-1}) \to L^1(\gamma_{-1})$ by Theorem~\ref{ImpowH1}. \emph{A fortiori}, it is then bounded $X^1_\mu(\gamma_{-1}) \to L^1(\gamma_{-1})$ for every $\mu\geq 0$.
\end{proof}

\subsection{Proof of Theorem~\ref{teo:X1Riesz}}
\begin{lemma}\label{Lemma4B}
For every $\lambda \geq 1$ there exists $C>0$ such that for every ball $B\in \mathcal{B}_1$
\[\|\nabla (\As+\lambda I)^{1/2} f\|_{L^1((4B)^c)}\leq \frac{C}{r_B^2}\|f\|_{L^1(B)} \quad \forall \, f\in L^1(\gamma_{-1}), \; \supp f\subseteq \bar{B}.\]
\end{lemma}
\begin{proof}
Theorem~\ref{teo:complex_powers} provides the kernel of the translated square root $(\As+\lambda I)^{1/2}$. Then, it is not hard to see that
\begin{align*}
\|\nabla(\As+\lambda I)^{1/2} f\|_{L^1((4B)^c)} &\lesssim \int_{(4B)^c}\int_B \int_0^1 \frac{r^{n+\lambda-1}|x-ry| e^{-\frac{|rx-y|^2}{1-r^2}}}{(1-r^2)^{\frac{n+2}{2}}(-\log r)^{3/2}}  \, \dd r |f(y)| \, \dd \gamma_{-1}(y) \, \dd x\\& = \int_B I(y)|f(y)|\, \dd \gamma_{-1}(y)\,
\end{align*}
where for $y\in B$
\begin{equation}\label{integrale4B}
I(y)= \int_0^1 \frac{r^{n+\lambda-1}}{(1-r^2)^{(n+2)/2}(-\log r)^{3/2}}  \int_{(4B)^c} |x-ry|e^{-\frac{|rx-y|^2}{1-r^2}}\, \dd x \, \dd r.
\end{equation}
We aim at showing that $I(y)\leq Cr_B^{-2}$ for every $y\in B$. We split $I(y)$ into $I_1(y) + I_2(y)$ according to the splitting $(0,1)= (0,3/4) \cup (3/4,1)$. Thus
\begin{align*}
I_1(y) \leq C \int_0^{3/4} \frac{r^{n+\lambda-1}}{(-\log r)^{3/2}}\int_{(4B)^c} |x-ry|e^{-c|rx-y|^2}\, \dd x \, \dd r.
\end{align*}
By the change of variables $rx-y=v$ in the inner integral and extending the integral over $(4B)^c$ to $\R^n$, we get
\begin{align*}
I_1(y) \leq C \int_0^{3/4} \frac{r^{\lambda-2}}{(-\log r)^{3/2}}\int_{\R^n} |v+(1-r^2)y|e^{-c|v|^2}\, \dd v \, \dd r.
\end{align*}
Now observe that, since $|y|\leq |c_B|+r_B \leq 2/r_B$,
\[|v+(1-r^2)y| \leq |v|+|y|\leq |v|+\frac{2}{r_B}\leq C\frac{|v|+1}{r_B}\]
and hence
\[
I_1(y) \leq \frac{C}{r_B} \int_0^{3/4} \frac{r^{\lambda-2}}{(-\log r)^{3/2}}\int_{\R^n} (|v|+1)e^{-c|v|^2}\, \dd v \, \dd r \leq \frac{C}{r_B} \int_0^{3/4} \frac{r^{\lambda-2}}{(-\log r)^{3/2}}\, \dd r\leq \frac{C}{r_B}
\]
for every $\lambda \geq 1$. Therefore, \emph{a fortiori}, $I_1(y)\leq Cr_B^{-2}$. Before looking at $I_2(y)$, we observe that for every $r\in (3/4,1)$, since $x-ry= (1-r^2)x + r(rx-y)$,
\begin{align*}
|x-ry|&\leq (1-r^2)|x|+r|rx-y|\\& \leq \frac{1-r^2}{r}|rx-y| +\frac{1-r^2}{r}|y| +r|rx-y| \leq C\bigg[|rx-y| + (1-r^2)|y|\bigg].
\end{align*}
Hence
\begin{align*}
I_2(y) &\leq C  \int_{3/4}^1 \frac{1}{(1-r^2)^{n/2 +2 }}\int_{(4B)^c} \bigg[\frac{|rx-y|}{\sqrt{1-r^2}} + \sqrt{1-r^2}|y|\bigg]e^{-\frac{|rx-y|^2}{1-r^2}}\, \dd x\, \dd r.
\end{align*}
Since $se^{-s^2}\leq C e^{-s^2/2}$ for $s>0$ and $e^{-s^2}\leq e^{-s^2/2}$, we get
\begin{align*}
I_2(y) & \leq C \int_{3/4}^1 \frac{1 + \sqrt{1-r^2}|y|}{(1-r^2)^{2 }}(1-r^2)^{-n/2}\int_{(4B)^c}e^{-\frac{|rx-y|^2}{2(1-r^2)}}\, \dd x\, \dd r.
\end{align*}
We now consider the cases when $r_{B,y}\geq 1$ and $r_{B,y}<1$ (recall Lemma~\ref{lemmatech} for the notation) separately. If $r_{B,y}\geq 1$, by Lemma~\ref{lemmatech}, 1.\ and 3.,
\begin{align*}
I_2(y) \leq \int_{3/4}^1 \frac{1 + \sqrt{1-r^2}|y|}{(1-r^2)^{2 }} \varphi \bigg(\frac{r_B}{\sqrt{1-r^2}}\bigg)\, \dd r
\end{align*}
that with the change of variable $r_B / \sqrt{1-r^2}=s$ turns out to be
\[I_2(y) \leq \frac{C}{r_B^2}\int_0^\infty (s+r_B |y|)\varphi(s)\, \dd s \leq \frac{C}{r_B^2}\int_0^\infty (s+1)\varphi(s)\, \dd s = \frac{C}{r_B^2}\]
since $r_B|y|\leq C$. If $r_{B,y}<1$, we split $(3/4,1)=(3/4, 1-r_{B,y}) \cup (1-r_{B,y},1)$ and $I_2(y) = I_2^1(y)+I_2^2(y)$ accordingly. Thanks to Lemma~\ref{lemmatech}, $I_2^2(y)$ can be treated exactly as we just did in the case $r_{B,y}\geq 1$, so it remains to consider only $I_2^1(y)$. By the change of variable $rx-y=v$ in the inner integral, we get
\begin{align*}
I_2^2(y)& \leq C \int_{3/4}^{1-r_{B,y}} \frac{1+\sqrt{1-r^2}|y|}{(1-r^2)^2} \int_{\R^n} e^{-|v|^2}\, \dd v\, \dd r
\\&\leq C \int_{3/4}^{1-r_{B,y}} \frac{1+\sqrt{1-r}|y|}{(1-r)^2}\, \dd r \leq C \bigg( \frac{1}{r_{B,y}} + \frac{|y|}{\sqrt{r_{B,y}}}\bigg)\leq \frac{C}{r_B^2}
\end{align*}
since $|y|\leq C/r_B$ and by the definition of $r_{B,y}$.
\end{proof}

For every $\lambda,\mu\geq0$ define the function
\[G_{\lambda,\mu}(z)\coloneqq\frac{z+\lambda}{z+\mu}=1 +\frac{\lambda-\mu}{z+\mu}\qquad z\geq 0,\]
and the corresponding multiplier for $\As$
\begin{equation}\label{GmunuA}
G_{\lambda,\mu}(\As) = (\As +\lambda I)(\As + \mu I)^{-1} = I + (\lambda-\mu)(\As +\mu I)^{-1}.
\end{equation}
\begin{lemma}\label{lemma4cose}
Let $\lambda,\mu\geq0$. Then
\begin{itemize}
\item[\emph{(1)}] $G_{\lambda,\mu}(\As)$ is bounded on $L^2(\gamma_{-1})$ for every $\lambda,\mu \geq 0$;
\item[\emph{(2)}] if $a$ is an $X^1_\mu$-atom supported in a ball $B$, then $\|G_{\lambda,\mu}(\As)\|_2^{-1} G_{\lambda,\mu}(\As)a$ is an $X^1_\lambda$-atom supported in $B$;
\item[\emph{(3)}] for every $f\in L^2(\gamma_{-1})$
\[(\Rs_\lambda)_j G_{\mu,\lambda}(\As)f = G_{\mu,\lambda}(\As -I)(\Rs_\lambda)_j f;\]
\item[\emph{(4)}] if $\lambda>1$, $G_{\mu,\lambda}(\As-I)$ is bounded on $L^1(\gamma_{-1})$ for every $\mu\geq 0$.
\end{itemize}
\end{lemma}

\begin{proof}
To prove (1), observe that for every $\lambda\geq 0$ and $\mu\geq 0$, $G_{\lambda,\mu}$ is bounded on $\sigma_2(\As) =\{n,n+1,\dots\}$ (recall Proposition~\ref{propspectrum}). 

To prove (2), notice that by Theorem~\ref{support:pres}
\[\supp G_{\lambda,\mu}(\As)\, a \subseteq \supp (\As +\mu I)^{-1}a \subseteq \bar{B}.\]
The size estimate is a consequence of the boundedness of $G_{\lambda,\mu}(\As)$ on $L^2(\gamma_{-1})$ proved in (1) and the size estimate on $a$. Finally, if $v\in q^2_\lambda(B)$,
\[(v, G_{\lambda,\mu}(\As)a) = ((\As +\lambda I)v, (\As +\mu I)^{-1}a) =0\]
since $(\As +\mu I)^{-1}a$ has zero mean with respect to $\gamma_{-1}$.

As for (3), it is easy to check that the identity holds when $f$ is one of the eigenfunctions of $\As$. Since the operators $(\Rs_\lambda)_j G_{\mu,\lambda}(\As)$ and $ G_{\mu,\lambda}(\As -I)(\Rs_\lambda)_j$ are  both bounded on $L^2(\gamma_{-1})$, the conclusion follows by density.\par

Finally, to prove (4), we recall that $\sigma_1(\As)=\{z\in \C\colon \Re z\geq 0\}$ by Proposition~\ref{propspectrum}.
Thus, if $\lambda>1$ and $\mu\ge 0$, the function 
\[G_{\mu,\lambda}(z-1)=\frac{z+\mu-1}{z+\lambda-1} = 1 +\frac{\mu-\lambda}{z+\lambda-1},\]
is bounded and holomorphic in a neighbourhood of $\sigma_1(\As)\cup\{\infty\}$. Hence the operator $G_{\mu,\lambda}(\As-I)$ is bounded on $L^1(\gamma_{-1})$ by~\cite[Theorem VII.9.4]{DunfordSchwartz}.
\end{proof}

\begin{proof}[Proof of Theorem~\ref{teo:X1Riesz}]
We begin by proving the unboundedness results, i.e., that $(\Rs_\lambda)_j$ is not bounded from $X^1_\mu(\gamma_{-1})$ to $L^1(\gamma_{-1})$ when $\lambda \in [0,1)$ and $\mu\ge 0$, and that $(\Rs_1)_j$ is not bounded from $X^1_\mu (\gamma_{-1})$ to $L^1(\gamma_{-1})$ if $\mu\not=1$. 

To do this, choose $\sigma_0\in S^{n-1}$ and let $B=B(0,1)$. By Lemma~\ref{Psi} (3),   there exists a function $\psi_0\in C^\infty_c(B)$ with $\int\psi_0\,\dd \gamma_{-1}=0$ such that $(\Psi_{\lambda,\sigma_0}, \psi_0)_{L^2(B,\gamma_{-1})}\neq 0$. By Corollary~\ref{corlnotm}, (1) the function $a_\mu=(\As+\mu I)\psi_0$ is a multiple of an $X^1_\mu$-atom. Then, by Lemma~\ref{Psi} (1)
\begin{align}\label{eqPsiatom}
(\Psi_{\lambda,\sigma_0}, a_\mu)_{L^2(B,\gamma_{-1})} =((\As +\mu I)\Psi_{\lambda,\sigma_0}, \psi_0)_{L^2(B,\gamma_{-1})}= (\mu-\lambda)(\Psi_{\lambda,\sigma_0}, \psi_0)_{L^2(B,\gamma_{-1})}.
\end{align}
If $\lambda\not=\mu$, then $(\Psi_{\lambda,\sigma_0}, a_\mu)_{L^2(B,\gamma_{-1})}\not=0$. Thus $(\Rs_\lambda)_j a_\mu \notin L^1(\gamma_{-1})$ for every $j=1,\dots,n$, $\lambda\in[0,1]$ and $\mu \neq \lambda$, by Lemma~\ref{lemmaunbounded} (i).
This proves both statements when $\lambda\not=\mu$. 
\par
It remains to consider the case $\lambda=\mu\in[0,1)$. Let $\psi_0$ and $a_\lambda$ be as before. Then $(\Psi_{\lambda,\sigma}, a_\lambda)_{L^2(B,\gamma_{-1})}=0$ for every $\sigma\in S^{n-1}$, by the same argument of~\eqref{eqPsiatom}. Moreover Lemma~\ref{Psi}, (1) yields
\[
(\Phi_{\lambda,\sigma_0}, a_\lambda)_{L^2(B,\gamma_{-1})} =((\As +\lambda I)\Phi_{\lambda,\sigma_0}, \psi_0)_{L^2(B,\gamma_{-1})}=2 (\Psi_{\lambda,\sigma_0}, \psi_0)_{L^2(B,\gamma_{-1})} \neq 0, 
\]
and the conclusion follows by Lemma~\ref{lemmaunbounded}, (ii).\par

\smallskip

We now prove that for every $j=1,\dots,n$ and $\lambda\geq 1$ the operator $(\Rs_\lambda)_j$ is uniformly bounded on $X^1_\lambda$-atoms, i.e.\ that if $\lambda\geq 1$
\begin{align}\label{uniform:bound:Riesz}
\sup \{ \|(\Rs_\lambda)_j a\|_1 \colon \, \mbox{$a$ is an $X^1_\lambda$-atom}\}<\infty.
\end{align}
The boundedness $(\Rs_\lambda)_j \colon X^1_\lambda(\gamma_{-1}) \to L^1(\gamma_{-1})$ will then follow by this and the weak type $(1,1)$ of $(\Rs_\lambda)_j$ proved in Theorem~\ref{teo:weaktype}, (ii), by the classical argument in~\cite[p.\ 95]{Grafakos}. The proof of~\eqref{uniform:bound:Riesz} follows essentially the same line as~\cite[Theorem 1.2]{B} (and of~\cite[Theorem 5.3]{MMV}).\par

Let $a$ be an $X^1_\lambda$-atom supported in an admissible ball $B$. Since
\[\|(\Rs_\lambda)_j a\|_{L^1}= \|(\Rs_\lambda)_j a\|_{L^1(4B)} + \|(\Rs_\lambda)_j a\|_{L^1((4B)^c)}\]
we estimate the two summands separately. By Cauchy-Schwartz inequality
\[\|(\Rs_\lambda)_j a\|_{L^1(4B)} \leq \gamma_{-1}(4B)^{1/2}\|(\Rs_\lambda)_j a\|_{L^2(4B)}\leq C \|a\|_2 \gamma_{-1}(4B)^{1/2} \leq C\]
where we used the boundedness of $(\Rs_\lambda)_j$ on $L^2(\gamma_{-1})$, the size property of $a$ and the local doubling property of $\gamma_{-1}$. As for the second term, we write 
\[(\Rs_\lambda)_j a= \partial_j (\As+\lambda I)^{1/2}[(\As+\lambda I)^{-1}a]. \]
Observe that $\supp (\As+\lambda I)^{-1}a \subseteq \bar{B}$, by Theorem~\ref{support:pres}. Thus, by Lemma~\ref{Lemma4B}
\[\|(\Rs_\lambda)_j a\|_{L^1((4B)^c)} \leq \frac{C}{r_B^2}\|(\As+\lambda I)^{-1}a\|_{L^1(B)}\]
so that it remains to estimate $\|(\As+\lambda I)^{-1}a\|_{L^1(B)}$. By Cauchy-Schwartz and Theorem~\ref{support:pres}
\[\|(\As+\lambda I)^{-1}a\|_{L^1(B)} \leq \gamma_{-1}(B)^{1/2}\|(\As+\lambda I)^{-1}a\|_{L^2(B)}\leq r_B^2.\]
This concludes the proof of~\eqref{uniform:bound:Riesz}. For what we said above, this implies (ii), and also (iii) when $\lambda=\mu$.\par

\smallskip

To complete the proof of (iii), let $\lambda>1$ and $\mu \geq 0$. Let $a$ be an $X^1_\mu$-atom. As before, it suffices to show that there exists a constant $C$ such that $\|(\Rs_\lambda)_j a\|_1 \leq C$. Since $G_{\lambda,\mu}(\As)G_{\mu,\lambda}(\As)=~I$,
\begin{align*}
(\Rs_\lambda)_j a &= G_{\mu,\lambda}(\As -I)\, (\Rs_\lambda)_j \, G_{\lambda,\mu}(\As) a 
\end{align*}
by Lemma~\ref{lemma4cose} (3). Thus
\[
(\Rs_\lambda)_j a= \|G_{\lambda,\mu}(\As)\|_{2,2}\, G_{\mu,\lambda}\,(\As -I)\, (\Rs_\lambda)_j\, \left[\|G_{\lambda,\mu}(\As)\|_{2,2}^{-1}\,G_{\lambda,\mu}(\As) a\right].
\]
Since $\|G_{\lambda,\mu}(\As)\|_{2,2}^{-1}\ G_{\lambda,\mu}(\As) a$ is an $X^1_\lambda$-atom by Lemma~\ref{lemma4cose} (2), $(\Rs_\lambda)_j$ is bounded from $X^1_\lambda$ to $L^1(\gamma_{-1})$ for what seen above and $G_{\mu,\lambda}(\As -I)$ is bounded on $L^1(\gamma_{-1})$ by Lemma~\ref{lemma4cose}, (4), the conclusion follows.
\end{proof}

\subsection*{Acknowledgements} It is my pleasure to thank Giancarlo Mauceri for his constant help, support and precious advice.


\begin{thebibliography}{9}

\bibitem{AdamsFournier}
R.\ Adams, J.\ J.\ F.\ Fournier, ``Sobolev spaces''. Second edition. Pure and Applied Mathematics (Amsterdam), 140. Elsevier/Academic Press, Amsterdam, 2003.

\bibitem{Arendt}	
W.\ Arendt,  ``Semigroups and evolution equations: functional calculus, regularity and kernel estimates''. Evolutionary equations. Vol. I, 1–85, Handb.\ Differ.\ Equ., North-Holland, Amsterdam, 2004.
	
\bibitem{Bakry}
D.\ Bakry, \emph{Étude des transformations de Riesz dans les variétés riemanniennes à courbure de Ricci minorée}, Séminaire de Probabilités, XXI, 137–172.

\bibitem{B}
T.\ Bruno, \emph{Endpoint results for the Riesz transform of the Ornstein-Uhlenbeck operator}, arXiv:1801.07214.

\bibitem{CarbDrag}
A.\ Carbonaro, O.\ Dragičević, \emph{Bellman function and linear dimension-free estimates in a theorem of Bakry.} J.\ Funct.\ Anal.\ 265 (2013), no.\ 7, 1085–1104.

\bibitem{CMM}
A.\ Carbonaro, G.\ Mauceri, S.\ Meda, \emph{$H^1$ and $BMO$ for certain locally doubling metric measure spaces.}, Ann.\ Sc.\ Norm.\ Super.\ Pisa Cl.\ Sci.\ (5) 8 (2009), no.\ 3, 543–582
	
\bibitem{CMMfinita}
A.\ Carbonaro, G.\ Mauceri, S.\ Meda, \emph{$H^1$ and $BMO$ for certain locally doubling metric measure spaces of finite measure}. Colloq.\ Math.\ 118 (2010), no.\ 1, 13–41.	
	
\bibitem{CoifmanWeiss}
R.\ R.\ Coifman, G.\ Weiss, \emph{Extensions of Hardy spaces and their use in analysis}. Bull.\ Amer.\ Math.\ Soc.\ 83 (1977), no.\ 4, 569–645.	
	
\bibitem{Davies1}
 E.\ B.\ Davies, ``Heat Kernel and Spectral Theory,'' Cambridge University Press, Cambridge, 1990.		
	
\bibitem{DunfordSchwartz}	
N.\ Dunford, J.\ T.\ Schwartz, ``Linear Operators, PART I: General Theory''. Inter-science Publishers LTD, London, 1958.

\bibitem{Evans}
L.\ C.\ Evans, ``Partial differential equations''. Second edition. Graduate Studies in Mathematics, 19. American Mathematical Society, Providence, RI, 2010.		
	
\bibitem{GMMST}
J.\ García-Cuerva, G.\ Mauceri, S.\ Meda, P.\ Sjögren, J.\ L.\ Torrea, \emph{Maximal operators for the holomorphic Ornstein-Uhlenbeck semigroup}. J.\ London Math.\ Soc.\ (2) 67 (2003), no.\ 1, 219–234.

\bibitem{GMST}
J.\ García-Cuerva, G.\ Mauceri, P.\ Sjögren, J.\ L.\ Torrea,  \emph{Higher-order Riesz operators for the Ornstein-Uhlenbeck semigroup}, Potential Anal.\ 10 (1999), no.\ 4, 379–407.  	

\bibitem{Grafakos}	
L.\ Grafakos, ``Modern Fourier analysis''. Third edition. Graduate Texts in Mathematics, 250. Springer, New York, 2014.
	
\bibitem{Grigor'yan}
A.\ Grigor'yan, ``Heat kernel and analysis on manifolds''. AMS/IP Studies in Advanced Mathematics, 47. American Mathematical Society, Providence, RI, 2009.		

\bibitem{Ketterer}
C.\ Ketterer, \emph{Ricci curvature bounds for warped products}. J.\ Funct.\ Anal.\ 265 (2013), no.\ 2, 266–299. 		
	
\bibitem{LiSjogrenWu}	
H-Q.\ Li, P.\ Sjögren, Y.\ Wu, \emph{Weak type $(1,1)$ of some operators for the Laplacian with drift}. Math.\ Z.\ 282 (2016), no. 3-4, 623–633.
	
\bibitem{MM}	
G.\ Mauceri, S.\ Meda, \emph{$BMO$ and $H^1$ for the Ornstein-Uhlenbeck operator}. J.\ Funct.\ Anal.\ 252 (2007), no.\ 1, 278–313. 
	 
	 \bibitem{MMS}
G.\ Mauceri, S.\ Meda, P.\ Sj\"ogren, \emph{Endpoint estimates for first-order Riesz transforms associated to the Ornstein-Uhlenbeck operator}. Rev.\ Mat.\ Iberoam.\ 28 (2012), no.\ 1, 77–91.

\bibitem{MMS1}
G.\ Mauceri, S.\ Meda, P.\ Sj\"ogren,  \emph{A maximal function characterization of the Hardy space for the Gauss measure}, Proc.\ Amer.\ Math.\ Soc.\ 141 (2013), no.\ 5, 1679–1692.	 
	 
\bibitem{MMVHardy}
G.\ Mauceri, S.\ Meda, M.\ Vallarino, \emph{Hardy-type spaces on certain noncompact manifolds and applications}. J.\ Lond.\ Math.\ Soc.\ (2) 84 (2011), no.\ 1, 243–268.		


\bibitem{MMVAtomic}
 G.\ Mauceri, S.\ Meda, M.\ Vallarino,  \emph{Atomic decomposition of Hardy type spaces on certain noncompact manifolds}. J.\ Geom.\ Anal.\ 22 (2012), no.\ 3, 864–891. 
		
		\bibitem{MMV}
G.\ Mauceri, S.\ Meda, M.\ Vallarino, \emph{Sharp endpoint results for imaginary powers and Riesz transforms on certain noncompact manifolds}, Studia Math.\ 224 (2014), no.\ 2, 153–168.
		
\bibitem{MMV-Harmonic}		
G.\ Mauceri, S.\ Meda, M.\ Vallarino, \emph{Harmonic Bergman spaces, the Poisson equation and the dual of Hardy-type spaces on certain noncompact manifolds.} Ann.\ Sc.\ Norm.\ Super.\ Pisa Cl.\ Sci.\ (5) 14 (2015), no.\ 4, 1157–1188.		 

\bibitem{MMVSymmetric}		
G.\ Mauceri, S.\ Meda, M.\ Vallarino, \emph{Higher order Riesz transforms on noncompact symmetric spaces}, arXiv:1507.04855.


\bibitem{MS}	 
S.\ Meda, F.\ Salogni, \emph{Harmonic analysis for the inverse Ornstein-Uhlenbeck semigroup}, preprint.
	 
	 
\bibitem{MenPerSor}
T.\ Menárguez, S.\ Pérez, F.\ Soria, \emph{The Mehler maximal function: a geometric proof of the weak type 1}. J.\ London Math.\ Soc.\ (2) 61 (2000), no.\ 3, 846–856.
	
	 
	 \bibitem{Metafune}
G.\ Metafune, \emph{$L^p$-spectrum of Ornstein-Uhlenbeck operators.} Ann.\ Scuola Norm.\ Sup.\ Pisa Cl.\ Sci.\ (4) 30 (2001), no.\ 1, 97–124.
	 
	 \bibitem{ONeill}
B.\  O'Neill, ``Semi-Riemannian geometry with applications to relativity''. Pure and Applied Mathematics, 103. Academic Press, Inc., New York, 1983.
	 
\bibitem{PerSor}
S.\ Pérez, F.\ Soria, \emph{Operators associated with the Ornstein-Uhlenbeck semigroup.} J.\ London Math.\ Soc.\ (2) 61 (2000), no. 3, 857–871.	 
	 
	 \bibitem{Salogni}
F.\ Salogni, \emph{Harmonic Bergman spaces, Hardy-type spaces and harmonic analysis of a symmetric diffusion semigroup on $\R^n$}. PhD Thesis, Università degli Studi di Milano-Bicocca, 2013. \url{https://boa.unimib.it/retrieve/handle/10281/41814/62217/phd_unimib_058626.pdf}

\bibitem{Sjogren}
 P.\ Sj\"ogren, \emph{Operators associated with the Hermite semigroup -- a survey}. Proceedings of the conference dedicated to Professor Miguel de Guzmán (El Escorial, 1996). J.\ Fourier Anal.\ Appl.\ 3 (1997), Special Issue, 813–823. 

\bibitem{Strichartz}
R.\ Strichartz, {\em Analysis of the Laplacian on the complete Riemannian manifold}, J.\ Funct.\ Anal.\ {\bf52} (1983), 48--79.
		
\end{thebibliography}
\end{document}